\newcount\Comments
\documentclass[letterpaper,10pt]{scrartcl}

\usepackage[utf8]{inputenc}
\usepackage[english]{babel}
\usepackage[automark]{scrlayer-scrpage}
\usepackage{amsmath,amstext,amsbsy,amsopn,amscd,amsxtra,upref,amssymb}
\usepackage{amsfonts,bm}
\usepackage{amsthm}
\usepackage{color}
\usepackage{graphicx}
\usepackage{sidecap}
\usepackage{multirow}
\usepackage{booktabs}
\usepackage{bm}
\usepackage{fullpage}

\usepackage{tabularx}
\usepackage[font=small]{caption}
\usepackage{subfig}

\usepackage{longtable}
\usepackage{rotating}

\usepackage{algorithm}
\usepackage{algpseudocode}
\extrafloats{100}

\usepackage{lineno,hyperref}

\graphicspath{{figures/}}
\DeclareGraphicsExtensions{.pdf,.eps,.png,.jpg,.jpeg}

\DeclareOldFontCommand{\bf}{\normalfont\bfseries}{\mathbf}

\newcommand{\bfSigma}{\boldsymbol \Sigma}

\newcommand{\bfm}{\boldsymbol \mu}

\newcommand{\R}{\mathbb{R}} 
\newcommand{\iid}{\stackrel{\text{i.i.d.}}{\sim}} 

\newcommand{\E}[2]{\mathbb{E}_{#2}\left[ #1 \right]} 
\newcommand{\var}[2]{\mathrm{Var}_{#1}\left[ #2 \right]} 
\newcommand{\samples}{m} 
\newcommand{\fid}{h} 
\newcommand{\mfis}{\hat{f}_{\fid,\samples}}
\newcommand{\params}{\boldsymbol \theta} 
\newcommand{\obs}{\boldsymbol y} 
\newcommand{\map}{\mathcal{F}} 
\newcommand{\chisq}[2]{\chi^2\left( #1\ ||\ #2 \right)} 
\newcommand{\KL}[2]{\mathrm{KL}\left( #1\ ||\ #2 \right)} 
\newcommand{\lapmean}{{\boldsymbol \mu}^{\mathrm{LAP}}} 
\newcommand{\lapcov}{{\boldsymbol \Sigma}^{\mathrm{LAP}}}
\newcommand{\noise}{\boldsymbol \Gamma} 
\newcommand{\pot}{\Phi} 
\newcommand{\ceil}[1]{\left\lceil #1 \right\rceil}

\newcommand{\zeros}{{\bf 0}}

\newcommand{\indicator}[1]{ {\bf 1} \left\{ #1 \right\}}

\newcommand{\cost}{c} 
\newcommand{\acc}{\delta} 
\newcommand{\hicost}{C}

\newcommand{\pr}{\pi_{\mathrm{pr}}} 
\newcommand{\prmean}{{\boldsymbol \mu}_{\mathrm{pr}}} 
\newcommand{\prcov}{{\boldsymbol \Sigma}_{\mathrm{pr}}}

\newcommand{\orlicz}[1]{\left\| #1 \right\|_{\psi_2}}
\newcommand{\dimx}{d} 
\newcommand{\dimy}{d'} 
\newcommand{\paramdomain}{\Theta} 
\newcommand{\spatialdomain}{\Omega} 

\newcommand{\bfx}{\boldsymbol x} 
\newcommand{\bfy}{\boldsymbol y}

\newcommand{\bfv}{\boldsymbol v}
\newcommand{\bfw}{\boldsymbol w}

\newcommand{\bA}{\boldsymbol A}
\newcommand{\eye}{\boldsymbol I} 
\newcommand{\cbv}{\boldsymbol e} 
\newcommand{\noisevar}{\boldsymbol \eta}

\newenvironment{keywords}
   {\begin{trivlist}\item[]{\bfseries\sffamily Keywords:}\ }
   {\end{trivlist}}

\newtheorem{lemma}{Lemma}
\newtheorem{theorem}{Theorem}
\newtheorem{corollary}{Corollary}

\newtheorem{assumption}{Assumption}
\newtheorem{proposition}{Proposition}
\newtheorem{remark}{Remark}

\ifpdf
\hypersetup{
  pdftitle={Context-aware surrogate modeling for balancing approximation and sampling costs in multi-fidelity importance sampling and Bayesian inverse problems},
  pdfauthor={Terrence Alsup and Benjamin Peherstorfer}
}
\fi

\title{Context-aware surrogate modeling for balancing approximation and sampling costs in multi-fidelity importance sampling and Bayesian inverse problems}

\author{Terrence Alsup\thanks{Courant Institute of Mathematical Sciences, New York University (\texttt{alsup@cims.nyu.edu}, \texttt{pehersto@cims.nyu.edu})}
\and Benjamin Peherstorfer\footnotemark[1]}

\date{\today}

\begin{document}

\maketitle

\begin{abstract}
Multi-fidelity methods leverage low-cost surrogate models to speed up computations and make occasional recourse to expensive high-fidelity models to establish accuracy guarantees. Because surrogate and high-fidelity models are used together, poor predictions by surrogate models can be compensated with frequent recourse to high-fidelity models. Thus, there is a trade-off between investing computational resources to improve the accuracy of surrogate models versus simply making more frequent recourse to expensive high-fidelity models; however, this trade-off is ignored by traditional modeling methods that construct surrogate models that are meant to replace high-fidelity models rather than being used together with high-fidelity models. This work considers multi-fidelity importance sampling and theoretically and computationally trades off increasing the fidelity of surrogate models for constructing more accurate biasing densities and the numbers of samples that are required from the high-fidelity models to compensate poor biasing densities.  Numerical examples demonstrate that such context-aware surrogate models for multi-fidelity importance sampling have lower fidelity than what typically is set as tolerance in traditional model reduction, leading to runtime speedups of up to one order of magnitude in the presented examples.
\end{abstract}

\begin{keywords}
Multi fidelity, importance sampling, Bayesian inverse problem, model reduction, Monte Carlo
\end{keywords}

\section{Introduction}
\label{section:introduction}
Surrogate models provide low-cost approximations of computationally expensive high-fidelity models and so are widely used to make tractable a variety of outer-loop applications such as control, optimization, and uncertainty quantification \cite{PWG}.  Typical examples of surrogate models are simplified-physics models \cite{NW,MG,CGHW},  data-fit and machine-learning models \cite{FK, RW}, and projection-based reduced models \cite{Antoulas, Rozza, modred_survey, Hesthaven,CQR2}.  Multi-fidelity methods combine surrogate models for speedups and high-fidelity models for accuracy guarantees \cite{PWG,NGX}. Recourse to the high-fidelity model enables compensation for poor surrogate accuracy, in stark contrast to traditional single-fidelity methods that use  surrogate models alone. The opportunity of multi-fidelity methods, which we exploit in the following, is that it is unnecessary that surrogate models achieve
tight accuracy guarantees because high-fidelity models are occasionally evaluated to correct results. Rather, it can be beneficial to use surrogate models with very low accuracy in favor of very cheap training and evaluation costs. Clearly, there is a limit of how low the accuracy of surrogate models can be in favor of costs before surrogate models become useless. Thus, in multi-fidelity approaches, there is a trade-off between increasing the accuracy of surrogate models with expensive
training methods versus making more frequent recourse to the expensive high-fidelity model to compensate less accurate, but cheaper, surrogate models. Surrogate models that exploit this trade-off are called context-aware models \cite{P}. This work derives context-aware surrogate models for multi-fidelity importance sampling (MFIS) estimators \cite{PCMW}, where the surrogate model is used for constructing a Laplace approximation as a biasing density. Our numerical results show that such context-aware surrogate models for MFIS can achieve an error reduction of more than one order of magnitude compared to using a single model alone.

We review related literature. First, there is work on adaptive discretizations for multi-level Monte Carlo methods and stochastic collocation methods \cite{10.1007/978-3-642-21943-6_10,ImplementationandanalysisofanadaptivemultilevelMonteCarloalgorithm,LANG2020109692} that adaptively refine meshes and time steps to obtain a non-uniform hierarchy of surrogate models. Additionally, there is work on continuous multi-level Monte Carlo \cite{doi:10.1137/18M1172259} that adapts the model hierarchy in a non-uniform fashion. In contrast to coarse-grid discretizations, we will consider surrogate models for constructing biasing densities, which incur training (offline) costs that we trade off with surrogate-model fidelity and frequency of recourse to the high-fidelity model. The work \cite{doi:10.1002/nme.4748} learns data-fit surrogate models for solving Bayesian inverse problems, without building on multi-fidelity methods and thus without deriving the trade-off between model accuracy and costs.  Second, the works \cite{P,IonutThesis} explore the trade-off between surrogate-model fidelity and number of times to make recourse to the high-fidelity for multi-fidelity Monte Carlo estimation with control variates, which is in contrast to using importance sampling for variance reduction as in this work. In \cite{DMSP}, the authors consider local, data-fit approximations and balance the decay rate of the bias due to the approximation with the variance of sampling with Markov chain Monte Carlo methods. Third, there is a large body of work on using surrogate models and multi-fidelity methods that build on importance sampling without explicitly exploiting the trade-off given by surrogate-model fidelity and frequency of recourse to the high-fidelity model. The work \cite{LX, LX2} develops a principled strategy to switch between sampling from a surrogate model and from the high-fidelity model to speedup failure and rare event probability estimation. In \cite{chen_accurate_2013}, the authors build on \emph{a posteriori} error estimators to decide if either a surrogate model or the high-fidelity model is evaluated.  The authors of \cite{doi:10.1137/17M1160069,doi:10.1137/19M1257433} develop a multi-fidelity method for importance sampling to efficiently estimate risk-measures such as the conditional value-at-risk.

We build on MFIS introduced in \cite{PCMW}. In particular, we develop bounds of the error of MFIS estimator that depends on the surrogate-model fidelity and then derive a trade-off between surrogate-model fidelity and computational costs. The first key ingredient is that we use a Laplace approximation computed with the surrogate model as biasing density. The quality of Laplace approximations has been studied in \cite{D} in terms of the Kullback-Leibler (KL) divergence and in \cite{SSW} in terms of the Hellinger distance when the noise level approaches zero. Instead, we consider the $\chi^2$ divergence \cite{nonpara} due to its natural interpretation as the variance of the importance weights.  
There is a large body of work on adaptive importance sampling that studies minimizing the $\chi^2$ divergence to derive an optimal biasing density \cite{SGDAIS,RyuBoydAIS,AkyildizMiguez}, but these works do not consider the cost of surrogate models during training.
The second key ingredient is bounding the error of the importance sampling estimator such as introduced in \cite{CD, APSS, SA}. These error bounds take the form of a probability divergence between the target distribution and the biasing distribution, which we will use to separate the error due to sampling from the error due to the quality of the biasing density that corresponds to the surrogate-model fidelity.

This manuscript is structured as follows. In Section~\ref{section:section2} we outline importance sampling in the multi-fidelity setting along with the bound on the mean-squared error (MSE) in terms of the $\chi^2$ divergence as presented in \cite{APSS}.  Section~\ref{section:section3} is the main contribution of this work and derives a bound on the $\chi^2$ divergence from the target to the biasing distribution in terms of the surrogate-model fidelity that leads to the formulation of an optimization problem for finding a trade-off.  In Section~\ref{section:section4}, we apply the results from Section~\ref{section:section3} in the case where the target distribution is a posterior distribution arising from a Bayesian inverse problem.  In Section~\ref{section:section5}, we demonstrate our method on two numerical examples. The proposed MFIS estimators with context-aware surrogate models achieve more than one order of magnitude error reduction compared to traditional importance sampling that uses the high-fidelity model alone with the same costs.

\section{Importance sampling and problem formulation}
\label{section:section2}

Section~\ref{section:section2p1} describes the setup of our problem.  Section~\ref{section:section2p2} is a brief overview of importance sampling and Section~\ref{section:section2p3} overviews how the quality of a biasing density influences importance sampling estimators in terms of the $\chi^2$ divergence.  Section~\ref{section:section2p4} illustrates the multi-fidelity approach to importance sampling and Section~\ref{section:section2p5} formulates the trade-off between fidelity and number of samples that we are interested in.

\subsection{Notation and problem setting}
\label{section:section2p1}

Let $(\paramdomain, \mathcal{B}(\paramdomain), p)$ denote a probability space where $\paramdomain = \R^{\dimx}$ is the domain for parameters $\params$, $\mathcal{B}(\paramdomain)$ is the Borel $\sigma$-algebra of $\paramdomain$, and $p$ is a probability distribution on $\paramdomain$.  Let $p$ be absolutely continuous with respect to the Lebesgue measure on $\R^d$ and refer to both the measure and the density function as $p$.  Furthermore, the density $p$ may only be known up to a normalizing constant $p = \frac{1}{Z} \tilde{p}$, where $\tilde{p} \ge 0$ is the un-normalized density and $Z = \int_{\params} \tilde{p}(\params)\ \mathrm{d}\params$ is the normalizing constant. In the following, we consider situations where the density $p$ and the un-normalized density $\tilde{p}$ are expensive to evaluate. The goal is to compute quantities of interest with respect to the target distribution $p$ which take the form of expectations
\begin{equation}
	\E{f}{p} = \int_{\paramdomain} f(\params) p(\params)\ \mathrm{d}\params,
\label{eq:qoi}
\end{equation}
where $f$ is a bounded measurable test function, i.e., $\|f\|_{L^{\infty}} < \infty$ where $\|f\|_{L^{\infty}} = \text{ess sup}_{\params \in \paramdomain} |f(\params)|$ under the measure $p$.

\subsection{Importance sampling}
\label{section:section2p2}

Let $q$ be another probability distribution on the Borel space\\ $(\paramdomain, \mathcal{B}(\paramdomain))$ that is absolutely continuous with respect to the Lebesgue measure on $\R^d$ and is such that $p$ is absolutely continuous with respect to $q$.  We let $q$ refer to both the probability distribution and the density function with respect to the Lebesgue measure. If sampling directly from $p$ is impossible and the normalizing constant $Z$ is unknown, then self-normalized importance sampling can be used with $q$ as the biasing distribution to estimate the expectation~\eqref{eq:qoi}.  Draw $\samples$ independent and identically distributed samples $\{\params^{(i)}\}_{i=1}^{\samples} \iid q$ from the biasing distribution $q$ and re-weight them with the target distribution $p$ to obtain the self-normalized importance sampling estimator
    \begin{equation}
    \hat{f}_{\samples} = \frac{\sum_{i=1}^{\samples} f(\params^{(i)}) w (\params^{(i)})}{\sum_{i=1}^{\samples} w(\params^{(i)}) }
    \label{eq:selfis}
    \end{equation}
of $\E{f}{p}$, where the importance weights $w(\params^{(i)})$ are given by evaluating the un-normalized likelihood ratio $w(\params) = \frac{\tilde{p}(\params)}{q(\params)}$ at the samples $\params^{(i)}$.  If all $w(\params^{(i)}) = 0$, then we define $\hat{f}_m = 0$.  The estimator \eqref{eq:selfis} is a consistent estimator of $\E{f}{p}$ as the sample size $\samples \to \infty$.

\subsection{Error of the importance sampling estimator}
\label{section:section2p3}

Theorem~2.1 of \cite{APSS} gives the following bound on the MSE of the self-normalized importance sampling estimator \eqref{eq:selfis}:  if $p$ is absolutely continuous with respect to $q$, then
 \begin{equation}
    \E{ \left( \hat{f}_{\samples} -\E{f}{p}\right)^2 }{} \
    \le\     \frac{4 \|f\|_{L^{\infty}}^2 }{\samples}\left(\chisq{p}{q} + 1\right)
    \label{eq:mse_bound}
    \end{equation}
holds, with the $\chi^2$ divergence from $p$ to $q$ defined as
 \begin{equation}
    \chisq{p}{q} + 1 = \var{q}{\frac{p}{q}} + 1 = \int_{\paramdomain} \left( \frac{p(\params)}{q(\params)} \right)^2 q(\params)\ \mathrm{d}\params  = \int_{\paramdomain} \frac{p(\params)}{q(\params)}\ p(\params)\ \mathrm{d}\params .
    \label{eq:chi2defn}
    \end{equation}
Note that the inequality~\eqref{eq:mse_bound} holds trivially if $\E{w^2}{q} = \infty$.  Since $f$ is bounded, it holds  $(\hat{f}_{\samples} - \E{f}{p})^2 \le 4 \|f\|_{L^{\infty}}^2$, which means that the bound \eqref{eq:mse_bound} is only useful if $\samples \ge \chisq{p}{q} + 1$.  The bound \eqref{eq:mse_bound} motivates setting the effective sample size to
    \begin{equation}
    \samples_{\mathrm{eff}} = \frac{\samples}{\chisq{p}{q} + 1},
    \label{eq:MEff}
    \end{equation}
    so that a large $\chi^2$ divergence corresponds to a large variance of the weights, meaning more samples are needed to reduce the MSE of the estimator \eqref{eq:selfis}.  The effective sample size \eqref{eq:MEff} motivates finding a biasing density $q$ that is close to $p$ with respect to the $\chi^2$ divergence.  The $\chi^2$ divergence is related to other probability divergences such as the Kullback-Leibler (KL) divergence
 \[
    	\KL{p}{q} = \int_{\paramdomain} \log\left( \frac{p(\params)}{q(\params)} \right)p(\params)\ \mathrm{d}\params
 \]
and the Hellinger distance
 \[
    	d_H(p,\ q) = \left( \frac{1}{2} \int_{\paramdomain} \left(\sqrt{p(\params)} - \sqrt{q(\params)} \right)^2\ \mathrm{d}\params  \right)^{1/2} .
\]
The relation is a lower bound given by Jensen's inequality
\[
    \mathrm{e}^{2d_H(p,\ q)^2} \le \mathrm{e}^{\KL{p}{q}} \le \chisq{p}{q} + 1\,,
 \]
see \cite{nonpara} for more general information regarding these probability divergences.

\subsection{Finding a biasing density}
\label{section:section2p4}

Let $(p_{\fid})_{h > 0}$ be a sequence of probability measures on $(\Theta, \mathcal{B}(\Theta))$, where the distributions $p_{\fid}$ are approximations to $p$ and the index $\fid > 0$ denotes the fidelity of the approximation. For each $\fid$, let $p_{\fid}$ be absolutely continuous with respect to the Lebesgue measure on $\R^d$ and use $p_{\fid}$ to denote both the density function and the distribution.  Let the density functions converge pointwise so $p_{\fid}(\params) \to p(\params)$ as $\fid \to 0$ for every $\params \in \paramdomain$.  Define $\hicost > 0$ as the cost of evaluating the un-normalized high-fidelity density $\tilde{p}$ and $\cost(\fid) > 0$ as the cost of evaluating the un-normalized surrogate density $\tilde{p}_{\fid}$.  The un-normalized surrogate densities $\tilde{p}_{\fid}$ can be used instead of $\tilde{p}$ to find a biasing density $q_{\fid}$ resulting in the multi-fidelity importance sampling (MFIS) \cite{PCMW} estimator
\begin{equation}
	 \mfis = \frac{\sum_{i=1}^{\samples} f(\params^{(i)}) w_{\fid} (\params^{(i)})}{\sum_{i=1}^{\samples} w_{\fid}(\params^{(i)}) } \quad \text{ where } \quad \{\params^{(i)}\}_{i=1}^{\samples} \iid q_{\fid},
\label{eq:mfis}
\end{equation}
of $\E{f}{p}$ with the importance weights $w_{\fid}(\params^{(i)}) = \tilde{p}(\params^{(i)})/\tilde{q}_{\fid}(\params^{(i)})$ given by the ratio of the un-normalized densities $\tilde{p}$ and $\tilde{q}_{\fid}$ at $\params^{(i)}$.  
Note that the un-normalized surrogate densities $\tilde{p}_h$ are not evaluated in computing the estimator~\eqref{eq:mfis} and are only evaluated when deriving the biasing density $q_h$.
The bound \eqref{eq:mse_bound} shows that the quality of the biasing density with respect to the MSE is determined by the variance of the weights $w_{\fid}(\params^{(i)})$ and thus that the number of samples needed to achieve an error tolerance depends directly on the fidelity $\fid$ of the surrogate density.

\subsection{Problem formulation}
\label{section:section2p5}
Multi-fidelity importance sampling gives rise to the following two-step process of estimating $\E{f}{p}$ for test functions $f$: (i) finding the biasing density $q_{\fid}$ from $p_{\fid}$ and (ii) evaluating the un-normalized densities $\tilde{q}_{\fid}$ and $\tilde{p}$ at $\samples$ samples to obtain an estimate \eqref{eq:mfis} of $\E{f}{p}$. Notice that $q_{\fid}$ is independent of the test function $f$ and thus can be re-used for many different test functions. The first step incurs training costs to derive $q_{\fid}$ using $p_{\fid}$, and the second step incurs online costs of evaluating the un-normalized surrogate and expensive high-fidelity densities. The two steps give rise to a trade-off: investing high training costs to find a good biasing density that keeps the $\chi^2$ divergence low means that fewer evaluations of the expensive high-fidelity density are required in the online step and vice versa. Traditional model reduction \cite{Rozza,modred_survey} typically targets computations where the surrogate model replaces the high-fidelity, where such a trade-off does not exist, instead of combining surrogate and high-fidelity models as in multi-fidelity methods such as MFIS. Thus, traditional model reduction provides little guidance on the mathematical formulation of this trade-off and the total costs.

\section{Context-aware surrogate models for multi-fidelity importance sampling} \label{section:section3}

We consider the following trade-off: given an error tolerance $\epsilon$, what is the optimal fidelity $\fid$ of the surrogate model that minimizes the total computational costs subject to the mean-squared error of the multi-fidelity importance sampling estimator \eqref{eq:mfis} being below or equal to the tolerance $\epsilon$. We refer to such surrogate models as \emph{context-aware} because the fidelity is determined specifically for the online computations of the problem (context) at hand \cite{P}, rather than being prescribed without taking the specific context of multi-fidelity computations into account as in traditional model reduction \cite{Rozza, modred_survey}.

Section~\ref{section:section3p1} revisits the notion of a sub-Gaussian distribution which is used in Section~\ref{section:section3p2} to derive an upper bound for $\chisq{p}{q_{\fid}}$ that depends on the fidelity $\fid$.  Section~\ref{section:section3p3} introduces a Laplace approximation $q_{\fid}$ of the low-fidelity surrogate density $p_{\fid}$ to be used as the biasing density and discusses its properties.  Section~\ref{section:section3p4} uses the bound~\eqref{eq:chi2bound} on the $\chi^2$ divergence to formulate an optimization problem that selects a fidelity $\fid^*$ based on the online stage of MFIS and derives the overall cost complexity of the corresponding estimator. Section~\ref{section:section3p5} summarizes the entire computational procedure in algorithmic form.

\subsection{Sub-Gaussian distributions}
\label{section:section3p1}

For importance sampling without a fixed test function $f$, it is imperative that the importance weights have finite variance (i.e., finite $\chi^2$ divergence) which means that the tails of the biasing density cannot be significantly lighter than the tails of the target density $p$.  Sub-Gaussian distributions are characterized by their fast tail decay.  A useful norm for quantifying the tail decay of a real-valued random variable, $X$, is the Orlicz norm defined as
\[
    \orlicz{X} = \inf \left\{ t > 0\ |\ \E{\exp(X^2/t^2)}{} \le 2 \right\}\,,
\]
see \cite[Sec.~2.5, Sec.~3.4]{HDP} for other equivalent definitions.  For a real random vector $\bfx = (x_1,\ldots,x_{\dimx})$, the Orlicz norm is defined to be
\[
    \orlicz{\bfx} = \sup_{\bfv \in S^{\dimx-1}} \orlicz{\bfv^T \bfx},
\]
where $S^{\dimx - 1} \subset \R^{\dimx}$ is the unit sphere defined as $S^{\dimx - 1} = \{ \bfv \in \R^{\dimx} : \|\bfv\|_2 = 1\}$.  A probability distribution $\pi$ is said to be sub-Gaussian if any random variable $\bfx \sim \pi$ has $\orlicz{\bfx} < \infty$.  Two examples of sub-Gaussian distributions are multivariate Gaussians and distributions with compact support.  If $\bfx \sim N(0,\sigma^2 \eye)$ then $\orlicz{\bfx} \le \sqrt{2}\sigma$.  In the following Lemma~\ref{lemma:subgaussian} we give a characterization of sub-Gaussian distributions that will be used in the following sections.  The lemma is a multi-dimensional version of Proposition 2.5.2 (iv) in \cite{HDP}.  We did not find this specific result in the literature and so we provide a proof in Appendix~\ref{appdx:proof}, even though it is a technical auxiliary result for us only.

        \begin{lemma}
A random vector $\bfx$ with density $\pi$ is sub-Gaussian if and only if there exists a symmetric positive-definite matrix ${\boldsymbol A}$ such that for all vectors $\bfm \in \mathbb{R}^{\dimx}$
    \[
    \E{ \exp\left( (\bfx - \bfm)^T{\boldsymbol A}(\bfx - \bfm) \right)}{ \pi} < \infty\, .
    \]
    \label{lemma:subgaussian}
    \end{lemma}

\begin{remark}
In the case where $\pi$ is a Gaussian with covariance $\bfSigma$, the matrix $\bA$ must be such that $\frac{1}{2}\bfSigma^{-1} - \bA$ is symmetric positive definite.  This constraint on $\bA$ will translate to a constraint on the biasing density for non-Gaussian target densities as will be made precise in the next section.
\end{remark}

\subsection{Bounding the $\chi^2$ divergence}
\label{section:section3p2}

In this section we derive the dependence of the MSE of the estimator~\eqref{eq:mfis} with respect to $\E{f}{p}$ on the fidelity $\fid$ used to find the biasing density $q_{\fid}$.  We bound $\chisq{p}{q_{\fid}}$ with respect to $\fid$ and we want this bound to factor into a part depending only on the ratio $p/p_{\fid}$ and a part depending only on the ratio $p_{\fid}/q_{\fid}$.  The following example demonstrates that such a decomposition is not straightforward: let
    \[
        p(x) = a\mathrm e^{-ax},\quad p_{\fid}(x) = b\mathrm e^{-bx},\quad q_{\fid}(x) = c\mathrm e^{-cx}\, \qquad x \ge 0 \, ,
    \]
    for $a,b,c > 0$, be three exponential distributions.  Then
    \[
        \chisq{p}{p_{\fid}} = \int_{0}^{\infty} \frac{a^2}{b} \mathrm e^{-(2a - b)x}\mathrm dx
        = \frac{a^2}{b(2a - b)}
    \]
    if $a > b/2$ and $\infty$ otherwise.  By taking $a = 2$, $b = 3/2$ and $c = 1$, we have that
    \[
        \chisq{p}{p_{\fid}} < \infty, \quad
        \chisq{p_{\fid}}{q_{\fid}} < \infty,
        \]
        but that
        \[
        \chisq{p}{q_{\fid}} = \infty\,,
    \]
    which means that we cannot directly decompose the $\chi^2$ divergence into the product of $\chi^2$ divergences with an intermediate distribution.  In contrast, the Cauchy-Schwarz inequality gives
 \begin{equation}
    \chisq{p}{q_{\fid}} + 1 = \left\| \frac{p}{q_{\fid}} \right\|_{L^1(p)} = \left\langle \frac{p}{p_{\fid}},\ \frac{p_{\fid}}{q_{\fid}}  \right \rangle_{L^2(p)} \le \left\| \frac{p}{p_{\fid}} \right\|_{L^2(p)} \left\| \frac{p_{\fid}}{q_{\fid}} \right\|_{L^2(p)}\,,
\label{eq:decomp}
\end{equation}
   which requires the likelihood ratios $p/p_h$ and $p_h/q_h$ to be in $L^2(p)$ as opposed to $L^1(p)$, which is required for the bound \eqref{eq:mse_bound} to hold and be finite.  The next four assumptions are sufficient for the likelihood ratios $p/p_h$ and $p_h/q_h$ to be in $L^2(p)$ and to decompose the $\chi^2$ divergence as in the right-hand side of Equation~\eqref{eq:decomp}.

    \begin{assumption}[Exponential form of the densities]
    The densities $p, p_h$ and $q_h$ have the form
    \[
    p(\params) = \frac{1}{Z}\mathrm{e}^{-\pot(\params)},
    \quad p_{\fid}(\params) = \frac{1}{Z_{\fid}} \mathrm{e}^{-\pot_{\fid}(\params)},
    \quad q_{\fid}(\params) = \frac{1}{\tilde{Z}_h} \mathrm{e}^{-\tilde{\pot}_{\fid}(\params)},
    \]
    with potentials $\pot, \pot_{\fid}, \tilde{\pot}_{\fid} \in \mathcal{C}^2(\Theta)$ that are twice differentiable and have continuous derivatives, normalizing constants $Z,Z_h,\tilde{Z}_h$, and $\pot_{\fid}(\params) \to \pot(\params)$ for all $\params \in \paramdomain$ as $h \to 0$.
    \label{assumption:potential}
    \end{assumption}

    \begin{assumption}[Decay of the target density]
        The target density $p$ is sub-Gaussian with matrix $\bA$; see Lemma~\ref{lemma:subgaussian}.
                \label{assumption:subgaussian}
    \end{assumption}

    \begin{assumption}[Error of the surrogate potentials]
    There exists an error function $\acc(\fid) > 0$ and a function $\tau(\params) \ge 0$, such that
    \[
        \pot_{\fid}(\params) \le \pot(\params)  + \acc(\fid) \tau(\params)
    \]
    for all $\params \in \paramdomain$, where $\delta(\fid) \to 0$ as $\fid \to 0$.      	
    \label{assumption:hi2low}
    \end{assumption}

    \begin{assumption}[Biasing densities]
    There exists a function $\gamma(h) > 0$ and a function $\omega(\params) \ge 0$ such that for all $\fid$
    \[
         \tilde{\pot}_{\fid}(\params) \le \Phi_{\fid}(\params) +  \gamma(\fid) \omega(\params)
    \]
    for all $\params \in \paramdomain$.
    \label{assumption:surrogate2approx}
    \end{assumption}

\begin{remark}
Assumption~\ref{assumption:surrogate2approx} does \emph{not} assume that $\gamma(\fid) \to 0$ as $\fid \to 0$.  Starting with Section~\ref{section:section3p3}, we will choose the density $q_h$ to be a Laplace approximation of $p_h$, which does not necessarily converge to $p_h$ as $h \to 0$. \label{rm:GammaNotZero}
\end{remark}

Theorem~\ref{theorem:chi2bound} gives the decomposition and bound depending on the fidelity $\fid$.
    \begin{theorem}
   Let Assumptions \ref{assumption:potential}, \ref{assumption:subgaussian}, \ref{assumption:hi2low}, and  \ref{assumption:surrogate2approx} hold and assume there exist constants $\tau_0,\omega_0 > 0$ such that
    \[
        \tau(\params) \le \|\params\|^2 + \tau_0, \quad 
        \omega(\params) \le \|\params\|^2 + \omega_0 \, .
    \]
    Let $\fid_{\max}$ be such that for all $\fid \le \fid_{\max}$
    \[
          \gamma(\fid) \le \frac{1}{4} \lambda_{\min}^{\bA} \, ,
    \]
    with $\bA$ being the matrix from Assumption~\ref{assumption:subgaussian} and $\lambda_{\min}^{\bA}$ being its smallest eigenvalue, then for all $\fid$ sufficiently small we have that
    \begin{equation}
        \chi^2(p\ ||\ q_{\fid}) + 1  \le K_0 \mathrm e^{ K_1 \delta(\fid) + K_2 \gamma(\fid)  }
        \label{eq:chi2bound}
    \end{equation}
    where $K_0,K_1,K_2$ are all positive constants independent of $\fid$. 
        \label{theorem:chi2bound}
    \end{theorem}
By the assumption in Theorem~\ref{theorem:chi2bound} that $\gamma(h) \le \lambda_{\min}^{\bA}/4$, the bound~\eqref{eq:chi2bound} can be written in the form
\begin{equation}
\chi^2(p\ || \ q_{\fid}) + 1 \leq \tilde{K}_0 \mathrm e^{K_1 \delta(h)}
\label{eq:chi2bound_simplified}
\end{equation}
where the constant $\tilde{K}_0$ now absorbs the dependency on the approximation $q_{\fid}$
\begin{equation}
    \tilde{K}_0 = K_0 \mathrm{e}^{K_2\lambda_{\min}^{\bA}/4} \ge K_0 \mathrm{e}^{K_2\gamma(h)}\, .
    \label{eq:constantK0}
\end{equation}
In the limit as the fidelity $h \to 0$, the upper bound~\eqref{eq:chi2bound_simplified} remains bounded by the constant $\tilde{K}_0$, which is determined entirely by the choice of biasing densities $q_h$.

    \begin{proof}[Proof of Theorem~\ref{theorem:chi2bound}]
   By Assumption~\ref{assumption:subgaussian}, $p$ is sub-Gaussian with matrix $\bA \succ 0$ so that by Lemma~\ref{lemma:subgaussian}
    \[
        \frac{1}{Z} \int_{\paramdomain} \exp \left( \params^T\bA \params  - \pot(\params) \right)\mathrm d\params < \infty \, .
    \]
    Recall that $Z$ is the normalizing constant from Assumption~\ref{assumption:potential}.

    \noindent\emph{Part 1: Bounding high-fidelity to surrogate ratio}

    The first term on the right-hand-side of Equation~\eqref{eq:decomp} can be bounded using Assumption~\ref{assumption:hi2low}:
 \begin{align*}
        \left\| \frac{p}{p_{\fid}} \right\|_{L^2(p)}^2
        &= \frac{1}{Z} \left( \frac{Z_{\fid}}{Z} \right)^2 \int_{\paramdomain} \exp\left\{ 2 \left( \pot_{\fid}(\params) - \pot(\params) \right) - \pot(\params)  \right\} \mathrm d\params\\
        &\le \frac{1}{Z} \left( \frac{Z_{\fid}}{Z} \right)^2 \int_{\paramdomain} \exp \left\{ 2 \delta(\fid) \left( \|\params\|^2 +  \tau_0 \right) - \pot(\params) \right\}\mathrm d\params \, .
 \end{align*}
 Re-writing this last line gives
 \begin{equation}
 	 \left\| \frac{p}{p_{\fid}} \right\|_{L^2(p)}^2 \le
	 \frac{1}{Z} \left( \frac{Z_{\fid}}{Z} \right)^2 \exp\left( 2 \tau_0 \delta(\fid) \right) \int_{\paramdomain} \exp \left\{ 2 \delta(\fid) \|\params\|^2  - \pot(\params) \right\}\mathrm d\params \, .
 \label{eq:part1}
 \end{equation}
Now the two dependencies of the right-hand side of~\eqref{eq:part1} on the fidelity $\fid$ are through the ratio $Z_{\fid}/Z$ and through $\delta(\fid)$. For now we just bound the integral on the right-hand side of~\eqref{eq:part1}, which is finite since $\bA \succ 2\delta(\fid) \eye$ for all $\fid$ sufficiently small.  Adding and subtracting $\params^T \bA \params$ in~\eqref{eq:part1} gives
    \begin{align*}
        \left\| \frac{p}{p_{\fid}} \right\|_{L^2(p)}^2
        &\le \frac{1}{Z} \left( \frac{Z_{\fid}}{Z} \right)^2 \exp\left( 2 \tau_0 \delta(\fid) \right) \int_{\paramdomain} \exp \left\{ 2 \delta(\fid) \|\params\|^2  - \pot(\params) \right\}\mathrm d\params \\
        &= \frac{1}{Z} \left( \frac{Z_{\fid}}{Z} \right)^2 \exp\left( 2 \tau_0 \delta(\fid) \right) \int_{\paramdomain} \exp \left\{ - \params^T \left(\bA - 2 \delta(\fid)\eye   \right)\params + \params^T \bA \params - \pot(\params) \right\}\mathrm d\params \, .
    \end{align*}
Putting this together with the fact that $\bA - 2\delta(\fid)\eye \succ 0$ gives
  \begin{equation}
       \left\| \frac{p}{p_{\fid}} \right\|_{L^2(p)}^2
       \le \frac{1}{Z} \left( \frac{Z_{\fid}}{Z} \right)^2 \exp\left( 2 \tau_0 \delta(\fid) \right) \int_{\paramdomain} \exp \left\{ \params^T \bA \params - \pot(\params) \right\}\mathrm d\params
 \label{eq:bound1}
\end{equation}
to complete the bound of the first term on the right-hand side of Equation~\eqref{eq:decomp}.

    \noindent\emph{Part 2: Bounding surrogate to biasing density ratio}

The second term on the right-hand side of Equation~\eqref{eq:decomp} is bounded in a similar fashion.  By Assumption~\ref{assumption:surrogate2approx} we can bound
    \begin{align*}
        \left\| \frac{p_{\fid}}{q_{\fid}} \right\|_{L^2(p)}^2
        &= \frac{1}{Z} \left( \frac{\tilde{Z}_{\fid}}{Z_{\fid}} \right)^2 \int_{\paramdomain}
        \exp \left\{ 2 \left( \tilde{\pot}_{\fid}(\params) - \pot_{\fid}(\params) \right) - \pot(\params) \right\}\mathrm d\params\\
        &\le \frac{1}{Z} \left( \frac{\tilde{Z}_{\fid}}{Z_{\fid}} \right)^2 \int_{\paramdomain}
        \exp \left\{ 2 \gamma(\fid) \left( \|\params\|^2 + \omega_0 \right)  - \pot(\params) \right\}\mathrm d\params\\
        &= \frac{1}{Z} \left( \frac{\tilde{Z}_{\fid}}{Z_{\fid}} \right)^2 \exp\left( 2 \omega_0 \gamma(\fid) \right) \int_{\paramdomain}  \exp \left\{ 2 \gamma(\fid) \|\params\|^2  - \pot(\params) \right\}\mathrm d\params \, .
      \end{align*}
Again we add and subtract $\params^T \bA \params$ to obtain
  \[
       \left\| \frac{p_{\fid}}{q_{\fid}} \right\|_{L^2(p)}^2   \le \frac{1}{Z} \left( \frac{\tilde{Z}_{\fid}}{Z_{\fid}} \right)^2 \exp\left( 2 \omega_0 \gamma(\fid) \right) \int_{\paramdomain} \exp \left\{ - \params^T \left(\bA - 2 \gamma(\fid)\eye   \right)\params + \params^T \bA \params - \pot(\params) \right\}\mathrm d\params \, .
 \]
 Using this with the fact that $\bA - 2\gamma(\fid)\eye \succeq 0$ for all $\fid \le \fid_{\max}$ gives
 \begin{equation}
          \left\| \frac{p_{\fid}}{q_{\fid}} \right\|_{L^2(p)}^2 \le \frac{1}{Z} \left( \frac{\tilde{Z}_{\fid}}{Z_{\fid}} \right)^2 \exp\left( 2 \omega_0 \gamma(\fid) \right) \int_{\paramdomain} \exp \left\{ \params^T \bA \params - \pot(\params) \right\}\mathrm d\params   \, .
          \label{eq:bound2}
\end{equation}

    Multiplying the right-hand sides of the bounds~\eqref{eq:bound1} and~\eqref{eq:bound2} and then taking the square root gives together with \eqref{eq:decomp} that
\begin{equation}
        \left\| \frac{p}{q_{\fid}} \right\|_{L^1(p)}
        \le \frac{1}{Z}\left( \frac{\tilde{Z}_{\fid}}{Z} \right) \exp \left\{\delta(\fid) \tau_0 + \gamma(\fid) \omega_0 \right\} \int_{\paramdomain} \exp \left\{ \params^T \bA \params - \pot(\params) \right\}\mathrm d\params
        \label{eq:bound3}
\end{equation}
holds. 
    The integral is independent of $\fid$, so it remains to bound the ratio of normalizing constants.

    \noindent\emph{Part 3: Bounding ratio of normalizing constants}

    In general, if $p_{\fid}$ is not in the family of biasing densities then we may have $\tilde{Z}_{\fid} \neq Z_{\fid}$, and thus,
    \[
        \frac{\tilde{Z}_{\fid}}{Z} \not \to 1
    \]
    as $\fid \to 0$.  Instead we just give a constant upper bound on $\tilde{Z}_{\fid}$ that is independent of the fidelity $\fid$.  By Assumption~\ref{assumption:potential}, the normalizing constant $\tilde{Z}_{\fid}$ satisfies
    \begin{align*}
        \tilde{Z}_{\fid} &= \int_{\paramdomain} \exp\left\{-\tilde{\pot}_{\fid}(\params)\right\}\mathrm d\params\\
        &= \int_{\paramdomain}
        \exp \left\{-\tilde{\pot}_{\fid}(\params) + \pot_{\fid}(\params) - \pot_{\fid}(\params) + \pot(\params) - \pot(\params) \right\}\mathrm d\params\\
        &= Z \int_{\paramdomain} \exp\left\{-\tilde{\pot}_{\fid}(\params) + \pot_{\fid}(\params) - \pot_{\fid}(\params) + \pot(\params) \right\}\ p(\params)\mathrm d\params \, .
    \end{align*}
    Dividing by $Z$ and using Assumptions~\ref{assumption:hi2low} and \ref{assumption:surrogate2approx} we have
\begin{equation}
    \frac{\tilde{Z}_{\fid}}{Z} \le \int_{\paramdomain} \exp \left\{ - \delta(\fid)(\|\params\|^2 + \tau_0) - \gamma(\fid)(\|\params\|^2 + \omega_0) \right\} p(\params)\ d\params
    \le 1 \, ,
\label{eq:bound4}
\end{equation}
because the term inside the exponential is less than or equal to 0 and $p$ is a density.    Finally, combining the bounds~\eqref{eq:bound1},~\eqref{eq:bound2}, and~\eqref{eq:bound4} gives the result
    \[
    \chisq{p}{q_{\fid}} + 1
    = \left\| \frac{p}{q_{\fid}} \right\|_{L^1(p)}
    \le  \exp \left\{\delta(\fid) \tau_0 + \gamma(\fid) \omega_0 \right\} \E{\exp\left( \params^T \bA \params \right) }{p} ,
    \]
    where the expectation is independent of $\fid$.  Here
    \[
    K_0 =  \E{\exp\left( \params^T \bA \params \right) }{p}, \quad K_1 = \tau_0, \quad K_2 = \omega_0
    \]
    are all independent of the fidelity $\fid$.
    \end{proof}

        \begin{remark}
        The assumption that $\tau(\params) \le \|\params\|^2 + \tau_0$ holds is similar to the pointwise Assumption 4.8 in Theorem 4.6 of \cite{S}.  In \cite{S}, the pointwise bound can grow faster with respect to $\params$ than in our case because there the Hellinger distance, which is upper-bounded by the $\chi^2$ divergence, is considered.
        \label{remark:pointwise}
        \end{remark}

\subsection{Laplace approximation}
\label{section:section3p3}

In the following, we use a Laplace approximation of a surrogate density $p_{\fid}$ as a specific choice of biasing density $q_{\fid}$.  A Laplace approximation $q_{\fid}$ is a Gaussian approximation to the density $p_{\fid}$ whose mean is a mode of $p_{\fid}$
\begin{equation}
    \lapmean_{\fid} = \underset{\params \in \paramdomain}{\mathrm{argmin}}\ -\log \tilde{p}_{\fid} (\params)
        = \underset{\params \in \paramdomain}{\mathrm{argmin}}\ \pot_{\fid}(\params) ,
\label{eq:laplacemean}
\end{equation}
and whose covariance is the negative inverse Hessian of the log-likelihood evaluated at the mode
\begin{equation}
    \lapcov_{\fid} = -\left[ \nabla \nabla^T \log \tilde{p}_{\fid}\left( \lapmean_{\fid} \right) \right]^{-1}
        = \left[ \nabla \nabla^T \pot_{\fid} \left( \lapmean_{\fid} \right) \right]^{-1} \, .
\label{eq:laplacecovariance}
\end{equation}

A Laplace approximation may not exist for certain distributions where the covariance matrix $\lapcov_h$ or Hessian at the mode is not full-rank.
If the following proposition applies, then a Laplace approximation exists and is a suitable biasing distribution; we refer to \cite{SSW} for in-depth discussions about Laplace approximations as biasing distributions if the covariance matrix is singular.

\begin{proposition}
Let Assumption~\ref{assumption:potential} hold and assume there exists a $\sigma_{\min}^2 > 0$, independent of $\fid$, such that 
\begin{equation}
    \params^T \lapcov_{\fid} \params \ge \sigma_{\min}^2 \|\params\|^2 \, ,
    \label{eq:collapse}
\end{equation}
for all $\params \in \paramdomain$.  Further, assume there exist constants $V \in \R$ and $v > 0$ such that
    \begin{equation}
        \pot_{\fid}(\params) \ge V - v \|\params\|^2
        \label{eq:laplace_pointwise}
    \end{equation}
    for all $\fid$.  Finally, let $B_R = \{\params : \|\params\| \le R\}$ be the ball of radius $R$ centered at $0$, and assume that for all $D > 0$, there exists an $R(D) > 0$ such that for all $\params \notin B_{R(D)}$ and all $h > 0$
        \begin{equation}
            \Phi_h(\params) \ge D \, .
            \label{eq:uniformgrowth}
        \end{equation}
    Then, the Laplace approximation satisfies Assumption \ref{assumption:surrogate2approx} for all $h$ sufficiently small.
    \label{prop:laplace}
    \end{proposition}

    \begin{proof}
    By Assumption~\ref{assumption:potential}, a Laplace approximation 
    \[
    \tilde{\pot}_{\fid}(\params) = \pot_{\fid}\left( \lapmean_{\fid} \right)
        + \frac{1}{2}\left(\params - \lapmean_{\fid}\right)^T \left[\nabla \nabla^T \Phi_{\fid}(\lapmean_{\fid}) \right]^{-1} \left( \params - \lapmean_{\fid} \right)
    \]
    is the second-order Taylor expansion of $\pot_{\fid}$ around one of the modes $\lapmean_{\fid}$.
    The first derivative is zero since it is expanded around a minimizer.  Therefore,
    \[
        \tilde{\pot}_{\fid}(\params) - \pot_{\fid}(\params) = -R_{\fid}(\params) \, ,
    \]
    where $R_{\fid}(\params)$ is the remainder of higher order terms from the Taylor expansion. 
The bound~\eqref{eq:collapse} implies that
\[
    \params^T \left( \lapcov_h \right)^{-1} \params \le \frac{1}{\sigma_{\min}^2} \| \params\|^2 \, ,
\]
and when combined with the bound~\eqref{eq:laplace_pointwise} gives
    \begin{align*}
        \tilde{\pot}_{\fid}(\params) - \pot_{\fid}(\params)
        & \le \tilde{\pot}_{\fid}(\params) - V + v \|\params\|^2\\
        &\le \pot_{\fid}\left( \lapmean_{\fid} \right)
        + \frac{1}{2\sigma_{\min}^2}\| \params - \lapmean_{\fid}\|^2 - V + v \|\params\|^2 \, .
        \end{align*}
        Combining this with the fact that $\|\bfx - \bfy\|^2 \le 2\|\bfx\|^2 + 2\|\bfy\|^2$ yields
\[
       \tilde{\pot}_{\fid}(\params) - \pot_{\fid}(\params)
        \le \pot_{\fid}\left( \lapmean_{\fid} \right) + \left( \frac{1}{\sigma_{\min}^2} + v \right)\|\params\|^2 + \frac{1}{\sigma_{\min}^2} \|\lapmean_{\fid}\|^2 - V \, .
\]

Now we claim that the terms $\Phi_h(\lapmean_h)$ and $\|\lapmean_h\|^2$ can be bounded independent of $h$.  Let $D = \Phi(0) + 1$ and consider that, by assumption, there exists a ball $B_{R(D)}$ such that 
\[
    \Phi_h(\params) \ge  \Phi(0)+1\, , \quad \forall \params \notin B_{R(D)} \, . 
\]
By Assumption~\ref{assumption:potential}, we know that $\Phi_h(0) \to \Phi(0)$ and so that for all $h$ sufficiently small, there exist points $\params_h'$, such that $\Phi_h(\params_h') \le \Phi(0)+1$.  Hence, the minimizers $\lapmean_h \in B_R$ for all $h$ sufficiently small.  Thus, there are constants $B_1,B_2 > 0$ independent of $h$ such that $\pot_{\fid}\left( \lapmean_{\fid} \right) \le B_1$ and $\|\lapmean_{\fid}\|^2 \le B_2$.  Thus, by setting
\[
    \gamma(h) = \frac{1}{\sigma_{\min}^2} + v,\quad \omega(\params) = \|\params\|^2 + \omega_0, \quad \omega_0 = \frac{B_1 + B_2/\sigma_{\min}^2 - V}{\sigma_{\min}^{-2} + v}\, 
\]
Assumption~\ref{assumption:surrogate2approx} holds.
\end{proof}

If Proposition~\ref{prop:laplace} applies, then it is guaranteed that there exists a Laplace approximation and that its covariance matrix remains non-singular as the fidelity $h$ is reduced: Condition~\eqref{eq:collapse} ensures that the covariance matrix $\lapcov_h$ is positive definite and hence that a Laplace approximation $q_h$ of $p_h$ exists for all $h > 0$.  The requirement that $\sigma_{\min}^2$ is independent of $h$ prevents the sequence of covariance matrices from approaching a singular matrix in the limit $h \to 0$. Condition~\eqref{eq:laplace_pointwise} is related to Assumption 2.6(i) from \cite{S}.  A pointwise bound is used to satisfy Assumption~\ref{assumption:surrogate2approx} and ensure the integrability from Theorem~\ref{theorem:chi2bound}. Condition~\eqref{eq:uniformgrowth} implies that $\Phi_h(\params) \to \infty$ as $\|\params\| \to \infty$ uniformly in $h$, and so we know that a global minimizer exists for each potential $\Phi_h$; however, it is not necessarily unique.  In the scenario where multiple global minima exist, we may choose any $\lapmean_h$ from the set of global minimizers.  In particular, we allow for multi-modal target and surrogate densities $p$ and $p_h$. 

\begin{remark}
If Proposition~\ref{prop:laplace} holds, then the Laplace approximation serves as a suitable biasing density for importance sampling in the sense that Assumption~\ref{assumption:surrogate2approx} holds, which is needed for Theorem~\ref{theorem:chi2bound}.  As the fidelity $h \to 0$ we may not have $\gamma(h) \to 0$ and so $\chi^2(p\ ||\ q_h)$ may not go to zero.
\end{remark}

\subsection{Trading off fidelity and costs of surrogate model for MFIS}
\label{section:section3p4}
We now consider the trade-off between selecting a fidelity $h$ to construct a Laplace approximation and the number of samples $m$ in the MFIS estimator \eqref{eq:mfis}. 

\subsubsection{Offline and online costs of MFIS with Laplace approximation as biasing density}
The total computational costs of estimating $\E{f}{p}$ with the MFIS estimator $\hat{f}_{h, m}$ defined in Equation~\eqref{eq:mfis} can be decomposed into training (offline) costs to fit the biasing density $q_{\fid}$ and the online costs to sample and re-weight; cf.~Section~\ref{section:section2p5}. 

In the training phase, the biasing density is constructed. In the following, we consider a Laplace approximation $q_h$ of the surrogate density $p_h$ as the biasing density. The Laplace approximation is constructed from $M$ evaluations of the un-normalized surrogate density $\tilde{p}_h$ and so the training costs are $Mc(h)$ in our case. Recall that $c(h)$ is the cost of evaluating the un-normalized surrogate density $\tilde{p}_h$. For example, in Section~\ref{section:section5}, $M$ will be the total number of surrogate-density evaluations used in Newton's method until machine precision is reached, where both the gradient and Hessian are computed using finite differences as well as computing the Hessian at the mode. 

In the online phase, the weights of the MFIS estimator are obtained by evaluating the target density and the biasing density at $m$ samples. We model the online costs as $m\hicost$, where $\hicost$ denotes the cost of a single evaluation of the un-normalized target density $\tilde{p}$. No evaluations of the surrogate density are necessary in the online phase because only the biasing density (Laplace approximation in our case) is evaluated, which has costs that typically are independent of $h$ and negligible compared to evaluating the target density $\tilde{p}$. However, notice that the online costs depend implicitly on the fidelity $h$ because the number of samples $m$ to reach an MSE below a threshold depends on the quality of the biasing distribution in the sense of the divergence $\chi^2(p || q_h)$; cf.~Section~\ref{section:section2p3}. 

We obtain as the total costs of the MFIS estimator 
\begin{equation}
    \mathrm{cost}(\hat{f}_{h,m}) = m\hicost + Mc(h)\,,
    \label{eq:mfiscost}
\end{equation}
which depends on the number of samples $m$ and on the fidelity $h$ of the surrogate.

\subsubsection{Cost complexity bounds of MFIS}
The following theorem provides cost-complexity bounds for the MFIS estimator under assumptions of the surrogate-models cost and error. We define the context-aware MFIS estimator to be the estimator~\eqref{eq:mfis} with fidelity $h^*$ and sample size $m^*$ given by the following theorem.

\begin{theorem}
Suppose that Theorem~\ref{theorem:chi2bound} and Proposition~\ref{prop:laplace} apply.  Consider a tolerance $0 < \epsilon \le 1$ and set $K_0' = 4\|f\|_{L^{\infty}}^2 \tilde{K}_0  + 1$, where $\tilde{K}_0$ is the constant in Equation~\eqref{eq:constantK0}.  If the surrogate density evaluation costs grow as $\cost(\fid) = \beta^{1/\fid}$ with the fidelity $h$ and the surrogate error decays as $\acc(\fid) = \alpha^{-1/\fid}$ in Assumption~\ref{assumption:hi2low}, with $\alpha,\beta > 1$, then there exist $h^* \in \mathbb{R}$ and $m^* \in \mathbb{N}$ such that the MFIS estimator $\hat{f}_{h^*,m^*}$ achieves an MSE less than the tolerance $\epsilon$ and the costs are bounded as
\[
     \mathrm{cost}(\hat{f}_{h^*,m^*}) \le  \overline{\mathrm{cost}}(\hat{f}_{h^*,m^*}) = \frac{\hicost K_0'}{\epsilon} \mathrm{e}^{ K_1 \epsilon^{1/(1 + \log_{\alpha}\beta)} } + M  \epsilon^{-1/(1 + \log_{\beta}\alpha)} \, .
\]
If instead $\cost(\fid) = \fid^{-\beta}$ and $\acc(\fid) = \fid^{\alpha}$, then the costs are bounded as
    \[
      \mathrm{cost}(\hat{f}_{h^*,m^*}) \le  \overline{\mathrm{cost}}(\hat{f}_{h^*,m^*}) =  \frac{\hicost K_0'}{\epsilon} \mathrm{e}^{K_1 \epsilon^{\alpha/(\alpha + \beta)} } + M \epsilon^{-\beta/(\alpha + \beta)} \, .
    \]
\label{thm:MainTheorem}
\end{theorem}

The rates on the error $\delta(h)$ and the cost $c(h)$ can arise, for example, in the Bayesian inverse problem setting in Section~\ref{section:section4}, where surrogate models are used to construct the surrogate densities $p_h$. Two concrete examples will be given in Section~\ref{section:section5}. Notice that $\gamma(h)$ from Assumption~\ref{assumption:surrogate2approx} influences implicitly the constant $\tilde{K}_0$ as shown in \eqref{eq:constantK0}, which amplifies Remark~\ref{rm:GammaNotZero} that it is unnecessary that $\gamma(h)$ goes to 0 for $h \to 0$ for Theorem~\ref{thm:MainTheorem} to hold.

Before we prove Theorem~\ref{thm:MainTheorem}, we state the following lemma that solves an auxiliary optimization problem highlighting the trade-off between the costs and fidelity of the surrogate model.

    \begin{lemma}
    Let $\hat{c}(\hat{\fid})$ and $\hat{e}(\hat{\fid})$ be continuous non-negative convex functions, which are not necessarily strictly convex. Let further $\hat{c}(\hat{\fid})$ decrease monotonically and $\hat{e}(\hat{\fid})$ increases monotonically as $\hat{\fid} \to \infty$. Let $\epsilon > 0$ be a tolerance and $\hat{M} \in \mathbb{N}$ be a constant independent of $\hat{\fid}$. Then, there exists a unique solution $(\hat{\fid}^*, \hat{\samples}^*)$ of
    
    \begin{equation}
    \begin{aligned}
    & \underset{\hat{\samples} \in \mathbb{R}, \hat{\fid} \ge 0}{\text{minimize}}
    & & \hat{\samples} \hicost  +  \hat{M} \hat{c}(\hat{\fid}) \\
    & \text{subject to}
    & & \frac{1}{\hat{\samples}}\hat{e}(\hat{\fid})  \le \epsilon\,.
    \end{aligned}
    \label{eq:optimization}
    \end{equation}

    \label{lmma:existenceandunique}
    \end{lemma}
    \begin{proof}[Proof of Lemma~\ref{lmma:existenceandunique}]
    For any $\hat{\fid}$, the optimal $\hat{\samples}$ is the one that achieves equality in the constraint
    \[
        \hat{\samples}
        = \frac{\hat{e}(\hat{\fid})}{\epsilon} \, .
    \]
    Plugging this into the objective function gives the minimization problem over $\hat{\fid}$ only.
    \begin{equation}
        \underset{\hat{\fid} \ge 0}{\mathrm{minimize}}\quad \hicost  \frac{\hat{e}(\hat{\fid}) }{\epsilon}  + \hat{M} \hat{c}(\hat{\fid}) \, .
        \label{eq:optimization2}
    \end{equation}
    Since the sum of convex functions is convex, we know that this objective function is convex.  Hence, if a minimum exists it is unique.

    We next show that the infimum of the objective function cannot occur as $\hat{\fid}\to \infty$ or as $\hat{\fid} \to 0$.  Since $\hat{c}(\hat{\fid})$ is non-negative and decreasing we know that $\hat{c}(\hat{\fid}) \to c_0$ for some constant $c_0 \ge 0$.  Moreover, $\hat{e}(\hat{\fid})$ is increasing, so we know that there exists an $\hat{h}_{\max} < \infty$, such that any optimal solution $\hat{\fid}^*$ must satisfy $\hat{\fid}^* \le \hat{\fid}_{\max}$.  
    Similarly, since $\hat{e}(\hat{\fid})$ is non-negative and decreasing as $\hat{\fid} \to 0$ we know that $\hat{e}(\hat{\fid}) \to e_0$
    for some constant $e_0 \ge 0$ as $\hat{\fid}\to 0$.  Moreover, $\hat{c}(\hat{\fid})$ is increasing as $\hat{\fid}\to 0$, and since the objective function~\eqref{eq:optimization2} is monotonically increasing as $\hat{\fid}\to 0$, we know that there exists an $\hat{\fid}_{\min} > 0$, such that any optimal solution $\hat{\fid}^*$ must satisfy $\hat{\fid}^* \ge \hat{\fid}_{\min}$.
    Hence
    \[
        \underset{\hat{\fid} \ge 0}{\mathrm{minimize}}\quad \hicost  \frac{\hat{e}(\hat{\fid}) }{\epsilon}  + \hat{M}\hat{c}(\hat{\fid}) = \underset{\hat{\fid} \in [\hat{\fid}_{\min}, \hat{\fid}_{\max}]}{\mathrm{minimize}}\quad \hicost  \frac{\hat{e}(\hat{\fid}) }{\epsilon}  + \hat{M} \hat{c}(\hat{\fid})
    \]
    Since the objective function is continuous over a compact set, we know that a minimizer exists.
    \end{proof}
    
\begin{proof}[Proof of Theorem~\ref{thm:MainTheorem}]
Combining the result of Theorem~\ref{theorem:chi2bound} in Equation~\eqref{eq:chi2bound_simplified} with the bound \eqref{eq:mse_bound},
let
\[
    e(h) = 4\|f\|_{L^{\infty}}^2 \tilde{K}_0 \mathrm{e}^{K_1 \delta(h)} \, .
\]
Because the composition of convex functions $\delta(h)$ and $x\mapsto \mathrm{e}^x$ is still convex, we know that $e(h)$ must be convex and therefore satisfies the assumptions of Lemma~\ref{lmma:existenceandunique}, meaning that a unique solution $\hat{h}^*,\hat{m}^* \in \mathbb{R}$ exists.

Consider $\cost(\fid) = \beta^{1/\fid}$ and $\acc(\fid) = \alpha^{-1/\fid}$ with $\alpha,\beta > 1$.   We can remove the constraint to instead minimize
\begin{equation}
	 \underset{\fid \ge 0}{\mathrm{minimize}}\quad \frac{4\|f\|_{L^{\infty}}^2 \tilde{K}_0 \hicost }{\epsilon} \mathrm{e}^{ K_1 \acc(\fid) }+ M \cost(\fid) ,
\label{eq:optimization_cais2}
 \end{equation}
which is analogous to \eqref{eq:optimization2}.  By setting the derivative of \eqref{eq:optimization_cais2} with respect to $\fid$ to zero, the optimal solution satisfies
    \[
        \frac{4\|f\|_{L^{\infty}}^2 \tilde{K}_0 K_1 \hicost \log \alpha}{M \log \beta} \mathrm{e}^{K_1 \alpha^{-1/h}} = \epsilon (\alpha \beta)^{1/h} \, ,
    \]
    meaning that $1/\hat{\fid}^* \in \mathcal{O}(\log_{\alpha \beta} \epsilon^{-1})$ as $\epsilon \to 0$ since the left-hand-side must approach a constant. Motivated by this observation, we set $1/h^* = \log_{\alpha\beta} \epsilon^{-1}$ exactly and then the number of samples needed is
    \[
    m^* = \ceil{\hat{m}^*}  = \ceil{\frac{4\|f\|_{L^{\infty}}^2 \tilde{K}_0}{\epsilon} \mathrm{e}^{K_1 \epsilon^{1/(1 + \log_{\alpha}\beta)}}  } \le \frac{4\|f\|_{L^{\infty}}^2 \tilde{K}_0}{\epsilon} \mathrm{e}^{ K_1 \epsilon^{1/(1 + \log_{\alpha}\beta)} } + 1 \, .
    \]
    where we have used that $\log_{\alpha\beta}\epsilon = \frac{\log_{\alpha} \epsilon}{1 + \log_{\alpha}\beta} = \frac{\log_{\beta} \epsilon}{1 + \log_{\beta}\alpha}$.  Since $\epsilon \le 1$ we know that $\mathrm{e}^{K_1 \epsilon^{1/(1 + \log_{\alpha}\beta)}}/\epsilon > 1$, and so 
    \[
        m^*  \le K_0' \frac{\mathrm{e}^{K_1\epsilon^{1/(1 + \log_{\alpha}\beta)}}}{\epsilon}\, .
    \]
    Plugging this in for $m$ into the objective function, gives an upper bound on the total computational costs
    \[
     \mathrm{cost}(\hat{f}_{h^*,m^*}) \le  \frac{\hicost K_0'}{\epsilon} \mathrm{e}^{ K_1 \epsilon^{1/(1 + \log_{\alpha}\beta)} } + M  \epsilon^{-1/(1 + \log_{\beta}\alpha)}  \, .
    \]

Now consider $\cost(\fid) = \fid^{-\beta}$ and $\acc(\fid) = \fid^{\alpha}$ with $\alpha,\beta \ge 1$. Set again the derivative to zero to find that the optimal solution satisfies
    \[
    \frac{4\|f\|_{L^{\infty}}^2\hicost \tilde{K}_0 K_1 }{M} \left( \frac{\alpha}{\beta} \right) \mathrm{e}^{ K_1 \fid^{\alpha} } \fid^{\alpha + \beta} = \epsilon \, ,
    \]
    so that $\hat{h}^* \in \mathcal{O}(\epsilon^{1/(\alpha + \beta)})$ as $\epsilon \to 0$.  If we set $h^* = \epsilon^{1/(\alpha + \beta)}$, then the number of samples needed is
    \[
    m^* = \ceil{\hat{m}^*} = \ceil{\frac{4\|f\|_{L^{\infty}}^2\tilde{K}_0}{\epsilon} \mathrm{e}^{K_1 \epsilon^{\alpha/(\alpha + \beta)} }} \le \frac{4\|f\|_{L^{\infty}}^2 \tilde{K}_0}{\epsilon} \mathrm{e}^{K_1 \epsilon^{\alpha/(\alpha + \beta)} } + 1 \le \frac{K_0'}{\epsilon} \mathrm{e}^{K_1 \epsilon^{\alpha/(\alpha + \beta)} }\, ,
    \]
    with total computational cost bounded as
    \[
    \mathrm{cost}(\hat{f}_{h^*,m^*}) \le \frac{\hicost K_0'}{\epsilon} \mathrm{e}^{K_1 \epsilon^{\alpha/(\alpha + \beta)} } + M \epsilon^{-\beta/(\alpha + \beta)}  \, .
    \]

\end{proof}

\begin{remark}
Although we have assumed that training costs correspond to fitting the Laplace approximation, Lemma~\ref{lmma:existenceandunique} shows that the results will extend more generally to any approximation where the costs of fitting the biasing density is convex in the fidelity $h$.
\end{remark}

\subsubsection{Discussion of cost complexity bounds of context-aware MFIS}
We now compare the cost bounds of the context-aware MFIS estimators $\hat{f}_{\fid^*, m^*}$ derived in Theorem~\ref{thm:MainTheorem} with the costs of fixed-fidelity MFIS estimators $\hat{f}_{\bar{h}, \bar{m}}$, where the fidelity $\bar{h}$ is fixed independent of $\epsilon$. The number of samples $\bar{m}$ is selected depending on the tolerance $\epsilon$ as
\[
    \bar{m} = \inf\left\{ m \in \mathbb{N} : \frac{e(\bar{h})}{m} \le \epsilon   \right\} \, ,
\]
analogously to the context-aware MFIS estimator. Note that the sample size depends as well on the fidelity $\bar{h}$.  The costs of the fixed-fidelity MFIS estimator are 
\[
\operatorname{cost}(\hat{f}_{\bar{h}, \bar{m}}) =  \bar{m}C + M c(\bar{h})\,.
\]
If $\delta(h) = \alpha^{-1/h}$ and $c(h) = \beta^{1/h}$, then the costs of the fixed-fidelity estimator are bounded as
    \[
     \mathrm{cost}(\hat{f}_{\bar{h},\bar{m}}) \le  \overline{\mathrm{cost}}(\hat{f}_{\bar{h},\bar{m}}) = \frac{\hicost K_0'}{\epsilon} \mathrm{e}^{  K_1 \alpha^{-1/\bar{h}} } + M \beta^{1/\bar{h}} \, ,
    \] 
and if $\delta(h) = h^{\alpha}$ and $c(h) = h^{-\beta}$ then the costs are bounded as
\[
    \mathrm{cost}(\hat{f}_{\bar{h},\bar{m}}) 
    \le \overline{\mathrm{cost}}(\hat{f}_{\bar{h},\bar{m}}) = \frac{ \hicost K_0'}{\epsilon} \mathrm{e}^{ K_1 \bar{h}^{\alpha} } + M \bar{h}^{-\beta} \, .
    \]
    
We now compare the costs of the context-aware MFIS and the fixed-fidelity MFIS estimators by comparing their cost upper bounds $\overline{\mathrm{cost}}$ as $\epsilon \to 0$.  First, consider the case where $\delta(h) = \alpha^{-1/h}$ and $c(h) = \beta^{1/h}$.  As $\epsilon \to 0$, we have that
\[
    \lim_{\epsilon \to 0} \frac{\overline{\mathrm{cost}}(\hat{f}_{\bar{h},\bar{m}})}{\overline{\mathrm{cost}}(\hat{f}_{h^*,m^*})} = \lim_{\epsilon \to 0} \frac{\frac{\hicost K_0'}{\epsilon} \mathrm{e}^{  K_1 \alpha^{-1/\bar{h}} } + M \beta^{1/\bar{h}}}{\frac{\hicost K_0'}{\epsilon} \mathrm{e}^{ K_1 \epsilon^{1/(1 + \log_{\alpha}\beta)} } + M  \epsilon^{-1/(1 + \log_{\beta} \alpha)}} \, . 
\]
Multiply the numerator and denominator by $\epsilon$ to get
\[
    \lim_{\epsilon \to 0} \frac{\hicost K_0' \mathrm{e}^{  K_1 \alpha^{-1/\bar{h}} } + \epsilon M \beta^{1/\bar{h}}}{\hicost K_0' \mathrm{e}^{ K_1 \epsilon^{1/(1 + \log_{\alpha} \beta)} } + M  \epsilon^{1-1/(1 + \log_{\beta}\alpha)}} \, .
\]
As $\epsilon \to 0$, the numerator goes to $CK_0' \mathrm{e}^{K_1 \alpha^{-1/\bar{h}}}$ and the denominator goes to $CK_0'$ since $\alpha > 1$.  Therefore, the speedup obtained with the context-aware MFIS estimator in the limit of $\epsilon \to 0$ is
\[
    \lim_{\epsilon \to 0} \frac{\overline{\mathrm{cost}}(\hat{f}_{\bar{h},\bar{m}})}{\overline{\mathrm{cost}}(\hat{f}_{h^*,m^*})} = \mathrm{e}^{K_1 \alpha^{-1/\bar{h}}} > 1 \, .
\]

Now consider the other case where $\delta(h) = h^{\alpha}$ and $c(h) = h^{-\beta}$.  We have that
\[
    \lim_{\epsilon \to 0} \frac{\overline{\mathrm{cost}}(\hat{f}_{\bar{h},\bar{m}})}{\overline{\mathrm{cost}}(\hat{f}_{h^*,m^*})} = \lim_{\epsilon \to 0} \frac{\frac{ \hicost K_0'}{\epsilon} \mathrm{e}^{ K_1 \bar{h}^{\alpha} } + M \bar{h}^{-\beta}}{\frac{\hicost K_0'}{\epsilon} \mathrm{e}^{K_1 \epsilon^{\alpha/(\alpha + \beta)} } + M \epsilon^{-\beta/(\alpha + \beta)} } \, .
\]
Multiplying both the numerator and denominator by $\epsilon$ gives
\[
    \lim_{\epsilon \to 0} \frac{ \hicost K_0' \mathrm{e}^{ K_1 \bar{h}^{\alpha} } + \epsilon M \bar{h}^{-\beta}}{\hicost K_0' \mathrm{e}^{K_1 \epsilon^{\alpha/(\alpha + \beta)} } + M \epsilon^{1-\beta/(\alpha + \beta)}} \, .
\]
As $\epsilon \to 0$, the numerator converges to $CK_0' \mathrm{e}^{K_1 \bar{h}^{\alpha}}$, and since $\beta/(\alpha + \beta) < 1$, the denominator converges to $CK_0'$.  Hence, the speedup obtained with the proposed context-aware MFIS estimator in the limit $\epsilon \to 0$ is
\[
    \lim_{\epsilon \to 0} \frac{\overline{\mathrm{cost}}(\hat{f}_{\bar{h},\bar{m}})}{\overline{\mathrm{cost}}(\hat{f}_{h^*,m^*})} = \mathrm{e}^{K_1 \bar{h}^{\alpha}} > 1 \, .
\]

In both cases we observe that as the tolerance $\epsilon \to 0$, the dominant term for the MFIS estimator cost approaches order $\mathcal{O}(1/\epsilon)$ and the bulk of the cost shifts to the online sampling cf.~Section~\ref{section:section2p5}.  We see that the speedup as $\epsilon \to 0$ depends on the rate of the error $\delta(\bar{h})$ going to zero.

\subsection{Computational procedure}
\label{section:section3p5}
Algorithm~\ref{alg:cais_procedure} summarizes the context-aware importance sampling procedure. Given constants $\tilde{K}_0,K_1,\hicost,M,\|f\|_{L^{\infty}}$, and the tolerance $\epsilon$ as well as the cost and accuracy functions $\cost$ and $\acc$, the context-aware importance sampling Algorithm~\ref{alg:cais_procedure} first solves the optimization problem~\eqref{eq:optimization} for $(\hat{\fid}^*, \hat{\samples}^*)$.  A Laplace approximation to the surrogate density $p_{\fid^*}$ is then computed using Newton's method.  In particular, we use the Newton-CG method where both the gradient and Hessian are computed using finite differences.  The Hessian at the mode is then inverted directly to obtain the covariance of the Laplace approximation.  This concludes the offline phase of finding the biasing density.  For the online phase we draw $\samples^* = \ceil{\hat{\samples}^*}$ samples from the Laplace approximation $q_{\fid^*}$ and re-weight using the un-normalized high-fidelity density $\tilde{p}$ using the estimator \eqref{eq:mfis}.

\begin{algorithm}[t]
\begin{algorithmic}[1]
\State{Constants $\tilde{K}_0,K_1,\hicost,\epsilon,M,\|f\|_{L^{\infty}}$ and functions $\cost,\acc$}
\State{Solve the optimization problem~\eqref{eq:optimization} for $(\fid^*, \hat{\samples}^*)$ using $\|f\|_{L^{\infty}}, \tilde{K}_0, K_1,\hicost,M,\epsilon,\cost,\acc$}
\State{Compute a Laplace approximation $q_{\fid^*}$ of $p_{\fid^*}$ with $M$ evaluations of $\tilde{p}_{\fid^*}$}
\State{Draw $\samples^* = \ceil{\hat{\samples}^*}$ i.i.d. samples $\{\params^{(i)}\}_{i=1}^{\samples^*}$ from $q_{\fid^*}$}
\State{Compute $\hat{f}_{\fid^*,\samples^*}$ using~\eqref{eq:mfis}}
\Return{Estimate $\hat{f}_{\fid^*,\samples^*}$}
\end{algorithmic}
\caption{Context-aware importance sampling}
\label{alg:cais_procedure}
\end{algorithm}

Algorithm~\ref{alg:cais_procedure} requires the constants $\tilde{K}_0, K_1, C, M,\|f\|_{L^{\infty}}$. Similar to other multi-level and multi-fidelity methods, we propose to first perform a pilot study to estimate these constants before using them in the computational procedure. Such pilot studies may be expensive; however, since the test function $f$ is independent of the constants, we only need to estimate these constants once and can then re-use them to compute a variety of statistics with respect to the target distribution $p$.

\section{Bayesian inverse problems}
\label{section:section4}

We now apply the context-aware MFIS estimator for inference in Bayesian inverse problems where the target $p$ is a posterior distribution and we are interested in expectations $\E{f}{p}$ of this distribution.  Section~\ref{section:section4p1} describes the general setup of a Bayesian inverse problem and Section~\ref{section:section4p2} applies the results of Section~\ref{section:section3} to the case where $p$ is a posterior distribution.

\subsection{Setup of a Bayesian inverse problem}
\label{section:section4p1}

   Let data $\obs \in \R^{\dimy}$ be generated by an unknown parameter $\params_{\mathrm{truth}} \in \R^{\dimx}$ with a Gaussian noise model,
   \[
   \obs = \map(\params_{\mathrm{truth}}) + \noisevar ,
   \]
   where $\noisevar \sim N(0,\noise)$, $\noise \in \R^{\dimy \times \dimy}$ is the covariance matrix (symmetric and positive definite) of the added noise, and $\map : \paramdomain \to \R^{\dimy}$ is the high-fidelity parameter-to-observable map.  Let $\pr$ denote a prior distribution over the parameter $\params$, so that the negative log-posterior has the form
    \[
   -\log p(\params) =  \pot(\params) = \frac{1}{2}\left\| \obs - \map(\params) \right\|_{\noise^{-1}}^2 - \log \pr(\params) \, .
    \]
    The norm is defined as $\|\bfv\|_{\noise^{-1}}^2 = \langle \noise^{-1} \bfv,\ \bfv \rangle$.  While it is possible to use the prior distribution as a biasing density, if the posterior contracts around the data then the $\chi^2$ divergence from the posterior to the prior may be very large resulting in a high variance estimator with a low effective sample size.

Let $\map_{\fid}$ denote the surrogate parameter-to-observable map with fidelity $\fid$ and let it be such that the series $\map_{\fid}(\params) \to \map(\params)$ converges pointwise for each $\params \in \paramdomain$.  Additionally, we assume that $\map,\map_h \in \mathcal{C}^2(\paramdomain)$.  In many cases the parameter-to-observable map $\map$ is a function of an intermediate state variable $u$, such as the full solution to a parametrized partial differential equation (PDE) depending on the parameters $\params$.  The surrogate parameter-to-observable map $\map_{\fid}$ is given by approximating this state variable $u$ with an approximation $u_{\fid}$.  The approximation for the state variable $u_{\fid}$ could be given by finite elements~\cite{FEM}, finite difference~\cite{leveque}, a different time step for an ordinary differential equation~\cite{leveque}, finitely many terms in a Karhunen-Lo{\`e}ve expansion~\cite{UQ}, and others.

We consider the case where the prior $\pr$ is Gaussian $N(\prmean, \prcov)$, so that we can write the potential from Assumption~\ref{assumption:potential} as
    \begin{equation}
    \pot(\params) = \frac{1}{2}\left\| \obs - \map(\params) \right\|_{\noise^{-1}}^2 + \frac{1}{2}(\params - \prmean)^T \prcov^{-1} (\params - \prmean) .
    \label{eq:bip_pot}
    \end{equation}
    With a Gaussian prior the resulting posterior distribution is always sub-Gaussian since we can take the matrix $\bA = \frac{1}{4} \prcov^{-1}$ in Lemma~\ref{lemma:subgaussian}.  The potentials $\pot_{\fid}$ are defined similarly but with the surrogate maps $\map_{\fid}$ replacing $\map$.

\subsection{Bounding $\chi^2$ divergence with model error}
\label{section:section4p2}

We now translate bounds on the model error between $\map$ and $\map_{\fid}$ to the $\chi^2$ divergence $\chisq{p}{q_{\fid}}$, where $q_{\fid}$ is a Laplace approximation to the surrogate posterior $p_{\fid}$.  The next two assumptions allow us to make the transition.

    \begin{assumption}
    The high-fidelity parameter-to-observable map $\map$ is globally Lipschitz meaning there exists a constant $B > 0$ such that for all $\params, \tilde{\params} \in \paramdomain$
    \[
   	 \|\map(\params) - \map(\tilde{\params})\| \le B \|\params - \tilde{\params}\| \, .
    \]
    \label{assumption:lipschitz}
    \end{assumption}

	Assumption~\ref{assumption:lipschitz} is almost the Lipschitz Assumption 2.7(ii) from \cite{S} except there the constant $B$ only needs to hold for bounded sets of $\params$.  Assumption~\ref{assumption:lipschitz} is satisfied if the map $\map$ is linear and smooth, for example.

    \begin{assumption}
    For all $\params \in \paramdomain$ and $\fid$ we have
    \[
    \|\map_{\fid}(\params) - \map(\params)\| \le \tilde{\delta}(\fid) \tilde{\tau}(\params)
    \]
    with $\tilde{\delta}(\fid) \to 0$ as $\fid \to 0$ with $\tilde{\tau}(\params)$ independent of $\fid$.
    \label{assumption:modelaccuracy}
    \end{assumption}

Assumption~\ref{assumption:modelaccuracy} is similar to Assumption (4.11) in Corollary 4.9 of \cite{S}, although the pointwise bound is also looser there than here for the same reason as given in Remark \ref{remark:pointwise}.  Theorem~\ref{theorem:bip_bound} is analogous to Theorem~\ref{theorem:chi2bound} from earlier but now is applied specifically to the Bayesian inverse problem.

    \begin{theorem}
    If Assumptions \ref{assumption:lipschitz} and \ref{assumption:modelaccuracy} are satisfied with $|\tilde{\tau}(\params)| \le \|\params\| + \tilde{\tau
    }_{0}$ for some $\tilde{\tau}_{0} > 0$, then Assumption \ref{assumption:hi2low} is also satisfied with
    \[
        \delta(h) = \left(\frac{2B + 1}{\kappa_{\min}}\right) \tilde{\delta}(h)
    \]
    and $\tau(\params)$ a quadratic function of $\|\params\|$ that is independent of $h$.
    \label{theorem:bip_bound}
    \end{theorem}

\begin{proof}
    Using the form of the log-posterior~\eqref{eq:bip_pot} we write
    \begin{align*}
        \left| \pot_{\fid}(\params) - \pot(\params) \right|
        &= \left| \|\map_{\fid}(\params) - \obs\|_{\noise^{-1} }^2 - \|\map(\params) - \obs\|_{\noise^{-1} }^2 \right|
    \end{align*}
    since the prior terms cancel. To simplify notation, set $\Delta(\params) = \map(\params) - \map_{\fid}(\params)$ and $\zeta(\params) = \map(\params) - \obs$, so that $\zeta(\params) - \Delta(\params) = \map_{\fid}(\params) - \obs$.  Now, we can instead write
\begin{align*}
        \left| \pot_{\fid}(\params) - \pot(\params) \right|
        &= \left| \|\zeta(\params)\|^2_{\noise^{-1}} - \|\zeta(\params) - \Delta(\params)\|_{\noise^{-1}}^2 \right|\\
        &= \left|  \|\zeta(\params)\|^2_{\noise^{-1}} -  \left \langle \noise^{-1}\left( \zeta(\params) - \Delta(\params) \right),\ \zeta(\params) - \Delta(\params)  \right \rangle  \right| \\
        &= \left| 2\langle \Delta(\params) ,\noise^{-1} \zeta(\params)\rangle - \|\Delta(\params)\|_{\noise^{-1}}^2 \right|\, .
\end{align*}
Applying the triangle inequality and then the Cauchy-Schwarz inequality to this last line gives
\begin{equation}
         \left| \pot_{\fid}(\params) - \pot(\params) \right|
         \le 2\|\Delta(\params)\|  \|\noise^{-1}\zeta(\params)\| + \|\Delta(\params)\|^2_{\noise^{-1}}\, .
\label{eq:thm4part1}
\end{equation}
Using that $\obs = \map(\params_{\mathrm{truth}}) + \noisevar$ and the triangle inequality gives
\begin{align*}
        \|\noise^{-1}\zeta(\params)\| &= \|\noise^{-1} (\map(\params) - \obs)\|\\
        &\le \|\noise^{-1} (\map(\params) - \map(\params_{\mathrm{truth}}))\| + \|\noise^{-1}\noisevar\|\, .
\end{align*}
Assumption~\ref{assumption:lipschitz} then gives the bound
\begin{equation}
	\|\noise^{-1}\zeta(\params)\| \le \frac{B}{\kappa_{\min}}\|\params - \params_{\mathrm{truth}}\| + \|\noise^{-1}\noisevar\| ,
	\label{eq:thm4part2}
\end{equation}
where $\kappa_{\min} > 0$ is the smallest eigenvalue of the covariance matrix $\noise$, i.e., the direction along which the posterior is most peaked.  Similarly, we bound
\begin{equation}
	\| \Delta(\params) \|^{2}_{\noise^{-1}} = \langle \noise^{-1}\Delta(\params),\ \Delta(\params) \rangle \le \frac{1}{\kappa_{\min}} \|\Delta(\params)\|^2\, .
\label{eq:thm4part3}
\end{equation}
Substituting bounds~\eqref{eq:thm4part2} and~\eqref{eq:thm4part3} into~\eqref{eq:thm4part1} yields
\[
        \left| \pot_{\fid}(\params) - \pot(\params) \right|
        \le 2\left(\frac{B}{\kappa_{\min}}\left( \|\params - \params_{\mathrm{truth}}\|  \right) + \|\noise^{-1}\noisevar\|\right)\|\Delta(\params)\| + \frac{1}{\kappa_{\min}}\|\Delta(\params)\|^2,
\]
 and the triangle inequality gives
     \begin{equation}
        \left| \pot_{\fid}(\params) - \pot(\params) \right|
        \le 2\left(\frac{B}{\kappa_{\min}}\left( \|\params\| + \|\params_{\mathrm{truth}}\|  \right) + \|\noise^{-1}\noisevar\|\right)\|\Delta(\params)\| + \frac{1}{\kappa_{\min}}\|\Delta(\params)\|^2 \, .
        \label{eq:thm4part4}
    \end{equation}

Assumption~\ref{assumption:modelaccuracy} along with the assumption that $|\tilde{\tau}(\params)| \le \|\params\| + \tilde{\tau}_{0}$ says $\|\Delta(\params)\|  \le \tilde{\delta}(\fid) \left( \|\params\| + \tilde{\tau}_{0} \right)$, so we get that
\begin{align*}
	\left| \pot_{\fid}(\params) - \pot(\params) \right|
        &\le 2\left(\frac{B}{\kappa_{\min}}\left( \|\params\| + \|\params_{\mathrm{truth}}\|  \right) + \|\noise^{-1}\noisevar\|\right) \tilde{\delta}(\fid) \left( \|\params\| + \tilde{\tau}_{0} \right)+ \frac{1}{\kappa_{\min}}\tilde{\delta}(\fid)^2 \left( \|\params\| + \tilde{\tau}_{0} \right)^2 \, ,
\end{align*}
and thus
\begin{align*}
    \left| \pot_{\fid}(\params) - \pot(\params) \right|
        \le& \left( \frac{2B}{\kappa_{\min}} \|\params_{\mathrm{truth}}\| + 2\|\noise^{-1}\noisevar\| + \frac{\tilde{\delta}(h) \tilde{\tau}_0}{\kappa_{\min}}  \right)  \tilde{\delta}(h) \tilde{\tau}_0 \\
        &+ \left( \frac{2B}{\kappa_{\min}}\tilde{\tau}_0 + \frac{2B}{\kappa_{\min}}\|\params_{\mathrm{truth}}\| + \frac{2}{\kappa_{\min}} \tilde{\delta}(h) \tilde{\tau}_0 + 2\|\noise^{-1}\noisevar\| \right) \tilde{\delta}(h) \|\params\|\\
        &+ \left( \frac{2B}{\kappa_{\min}} + \frac{1}{\kappa_{\min}} \tilde{\delta}(h) \right) \tilde{\delta}(h) \| \params\|^2\,.
\end{align*}
Using that $\tilde{\acc}(\fid) \le 1$ for all $\fid$ sufficiently small and $\|\params\| \le 1 + \|\params\|^2$ gives
\[
	\left| \pot_{\fid}(\params) - \pot(\params) \right| \le \delta(\fid)\tau(\params) ,
\]
where 
\[
    \delta(h) = \left(\frac{2B + 1}{\kappa_{\min}}\right) \tilde{\delta}(h) 
\]
is as in Assumption~\ref{assumption:hi2low} and $\tau(\params)$ is quadratic in $\|\params\|$ and is bounded independent of $h$.
\end{proof}

\begin{corollary}
Suppose that Theorem~\ref{theorem:bip_bound} and Proposition~\ref{prop:laplace} both apply.  Then the cost complexity of the context-aware importance sampling estimator with a Laplace approximation biasing density has cost complexity given by Theorem~\ref{thm:MainTheorem}.
\end{corollary}

\section{Numerical Results}
\label{section:section5}

This section demonstrates our context-aware importance sampling approach on two examples.  All runtime measurements were performed on compute nodes equipped with Intel Xeon Gold 6148 2.4GHz processors and 192GB of memory using a Python 3.6 implementation.

\subsection{Steady-state heat conduction}

In the first example we consider a steady-state heat diffusion model with constant heat source and infer a 6-dimensional variable diffusivity.

\subsubsection{Problem Setup}

Let $\spatialdomain = (0,1) \subset \R$ and $\paramdomain = \R^6$ and consider the PDE
\begin{align}
\begin{split}
	- \left( \exp\left( k(x; \params) \right) u_x(x; \params) \right)_x &= 1, \quad x \in \spatialdomain\\
	u(0; \params) &= 0\\
	k(1; \params) u_x(1; \params) &= 0
 \end{split}
\label{eq:laplace_pde}
 \end{align}
where $\params = (\theta_1,\ldots,\theta_6)^T \in \paramdomain$, $k:\spatialdomain \times \paramdomain \to \R$ is the log-diffusivity, and $u : \spatialdomain \times 	\paramdomain \to \R$ is the temperature function.  The log-diffusivity $k(x;\ \params)$ is a smoothed piecewise constant.  In particular, let
\[
	I(x, \alpha) = \left( 1 + \exp \left( - \frac{x - \alpha}{0.005} \right) \right)^{-1}
\]
and $\alpha_i = (i-1)/6$ for $i=1,\ldots,7$.  Define
\begin{equation}
	\hat{k}_i(x; \params) = (1 - I(x,\alpha_i)) \hat{k}_{i-1}(x; \params) + I(x,\alpha_i) \theta_i
\label{eq:log_diffusivity}
\end{equation}
for $i=2,\ldots,6$ and $\hat{k}_1(x; \params) = \theta_1$.  Now set the log-diffusivity $k = \hat{k}_6$.  We discretize \eqref{eq:laplace_pde} in the spatial domain $\spatialdomain$ using linear finite elements with mesh width $\fid > 0$ (i.e. $\fid^{-1}$ many elements) and the corresponding sparse (tri-diagonal) linear system is solved using a Cholesky factorization.  The parameter-to-observable map $\map_{\fid} : \paramdomain \to \R^{120}$ is the discretized solution $u_{\fid}$ with mesh width $\fid$ evaluated at 120 equally-spaced points on $\spatialdomain$
 \[
	 \left( \map_{\fid}(\params) \right)_i = u_{\fid}(i/120),\quad i = 1,\ldots,120\,.
 \]
For the high-fidelity parameter-to-observable map we set $H^{-1} = 256$ elements, (i.e. $\map = \map_{H}$) and for the surrogate maps $\map_{\fid}$ we consider $\fid^{-1} = 8,12,16,\ldots, 64$ (multiples of 4 for the number of elements).

\subsubsection{Setup of the inverse problem}

A single observation $\obs = \map(\params_{\mathrm{truth}}) + \noisevar$ is generated where $\params_{\mathrm{truth}} = (1,\ldots,1)^T \in \R^6$ and $\noisevar \sim N(\zeros,\ 10^{-5} \eye_{120 \times 120 })$.  The added noise corresponds to approximately 1\% of the true solution $u$ at the right endpoint $x = 1$.  The prior distribution is taken to be a Gaussian with mean $\prmean = (1,\ldots,1)^T \in \R^6$ and covariance $\prcov = 10^{-1}  \eye_{6 \times 6} \in \R^{6\times 6}$.  For the test function let $\bfv_1 \in \R^6$ be the largest eigenvector of $\lapcov$ and set
 \begin{equation}
        f(\params) = 2\cdot \indicator{  (\params - \lapmean) \cdot \bfv_1 \ge 0} - 1
        \label{eq:testfun}
 \end{equation}
so that $f(\params) \in \{ \pm 1\}$ for all values of $\params$.  The idea behind this choice of test function is that the asymptotic variance of the MFIS estimator \eqref{eq:mfis} is largest whenever $f$ is not tightly concentrated around its expectation under $q_{\fid^*}$.  Here the expectation of $f$ under $q_{\fid^*}$ should be close to zero even though $f$ itself is never close to zero.

\subsubsection{Results}

A Laplace approximation to each surrogate posterior $p_{\fid}$ is fit using the Newton-CG method where the gradient and Hessian matrix are computed using a second-order finite difference scheme with a total of $M = 1150$ model evaluations at each fidelity.  The cost function has the form $\cost(\fid) = c_0 + c_1/\fid$, where $c_0$ is included to model any baseline cost independent of the fidelity, and accuracy has the form $\acc(\fid) = a_1 \fid^2$ since we use linear finite elements.  The cost is linear in $\fid^{-1}$ since the system of linear equations is tri-diagonal. We estimate the $\chi^2$ divergence with Monte Carlo estimator
\begin{equation}
    \hat{\chi}^2_{\fid,\samples} = \samples \frac{ \sum_{i=1}^{\samples} \left( \tilde{p}(\params^{(i)}) /  q_{\fid}(\params^{(i)}) \right)^2  }{ \left( \sum_{i=1}^{\samples} \tilde{p}(\params^{(i)}) /  q_{\fid}(\params^{(i)}) \right)^2 } \longrightarrow \chisq{p}{q_{\fid}} + 1, \quad \text{almost surely as } \samples \to \infty
\label{eq:chi2estimate}
\end{equation}
and $\{ \params^{(i)} \}_{i=1}^{\samples}$ are i.i.d. samples drawn from $q_{\fid}$.  Then the curve $\tilde{K}_0 e^{K_1 \fid^2}$ is fit using the estimated $\chi^2$ values $\hat{\chi}^2_{\fid,10^3}$ for each fidelity $\fid^{-1} = 8,12,16,\ldots,64$ averaged over $N_1 = 500$ independent trials.  The measured $\chi^2$ values are
\[
\hat{\chi}^2_{\mathrm{meas},\fid} = \frac{1}{N_1} \sum_{i=1}^{N_1} \left(\hat{\chi}^2_{\fid,\samples}\right)^{(i)}
\]
with the superscript $(i)$ denoting one of the independent trials.  The fitted curve along with the measured values are shown in Figure~\ref{fig:laplace_chi2_optfid}. The $\chi^2$ divergence is large for low fidelities but quickly levels off and then is limited by the restriction of the biasing density to be the Laplace approximation rather than the surrogate density itself.

Since we only consider finitely many surrogate models $\fid^{-1} = 8,12,16,\ldots,64$, we approximate the solution of the optimization problem \eqref{eq:optimization} with a brute force search to find the best fidelity $\fid^*$ from the list of fidelities that we consider and set $\samples^* = \ceil{\hat{m}^*}$ with $\hat{m}^*$ corresponding to $h^*$.  Figure~\ref{fig:laplace_chi2_optfid} shows the selected fidelity as a function of the tolerance $\epsilon$.  As the tolerance shrinks we require a higher-fidelity model to fit a Laplace approximation.  Using the pair $(\fid^*,\samples^*)$, Figure~\ref{fig:laplace_tradeoff} shows the theoretical optimal trade-off between cost in seconds and the MSE (tolerance) of the estimator $\hat{f}_{\fid^*,\samples^*}$.  We estimated the true value $\E{f}{p}$ using $\hat{f}_{H,10^5}$ and averaged the results over $N_2 = 500$ independent trials (again denoted by the superscript $(i)$)
\begin{equation}
	\bar{f} = \frac{1}{N_2} \sum_{i=1}^{N_2} \hat{f}^{(i)}_{H,10^5}\,.
\label{eq:truth_estimate}
\end{equation}
Next we estimated the MSE of $\hat{f}_{\fid^*,\samples^*}$ using $N_3 = 1000$ trials
\begin{equation}
		\widehat{\mathrm{MSE}}_{\epsilon} = \frac{1}{N_3} \sum_{i=1}^{N_3} \left(  \hat{f}_{\fid^*, \samples^*}^{(i)} - \bar{f}  \right)^2\,.
\label{eq:mse_estimate}
\end{equation}
Here the subscript $\epsilon$ denotes the dependence of the pair $(\fid^*,\samples^*)$ on the tolerance $\epsilon$.  Figure~\ref{fig:laplace_tradeoff} shows the averaged MSE over $N_3 = 1000$ trials for different tolerances $\epsilon$ as well as the MSE for the estimators $\hat{f}_{H,\samples_H}$ and $\hat{f}_{\fid_0,\samples_{\fid_0}}$ where the number of samples is
\[
    \samples_{\fid} = \ceil{\frac{\tilde{K}_0}{\epsilon} \exp\left( K_1 \fid^{2} \right)}
\]
and $\fid_0 = 8$ is the lowest fidelity we consider (for the surrogate only estimator we average only $N_3 = 500$ trials).  For moderate error tolerances we can achieve an order of magnitude speedup since most of the cost comes from fitting a Laplace approximation; using a very accurate model is not necessary, but using a very cheap surrogate model is insufficient.  As the tolerance shrinks, most of the computation shifts to the online sampling phase which begins to dominate and little speedup is obtained. This matches the theoretical speedup derived in Section~\ref{section:section3p4}.

\begin{figure}
\begin{minipage}{0.49\textwidth}
{{\Large\resizebox{\columnwidth}{!}{\input{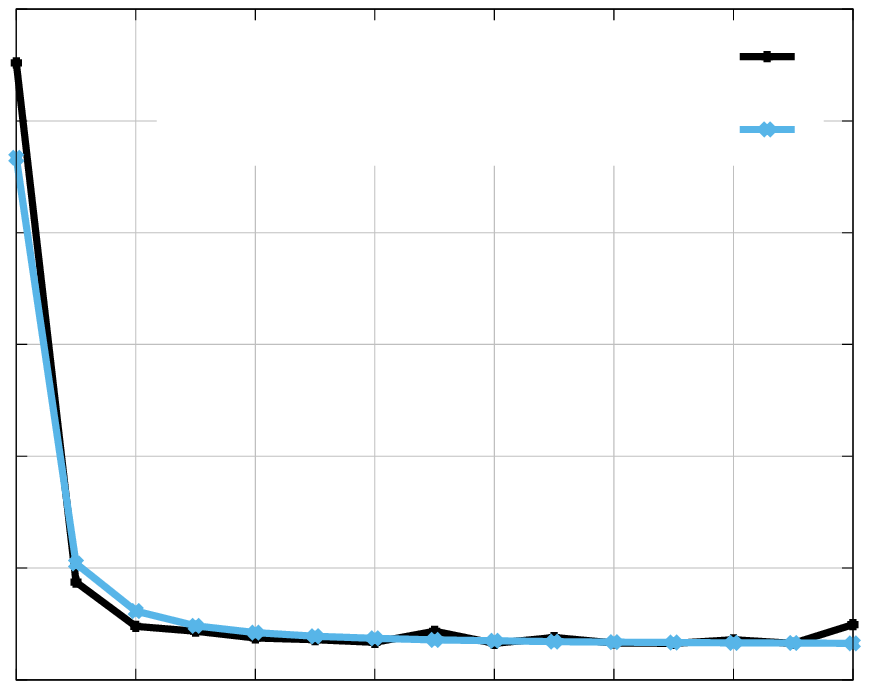}}}}
\end{minipage}
\hfill
\begin{minipage}{0.49\textwidth}
{{\Large\resizebox{\columnwidth}{!}{\input{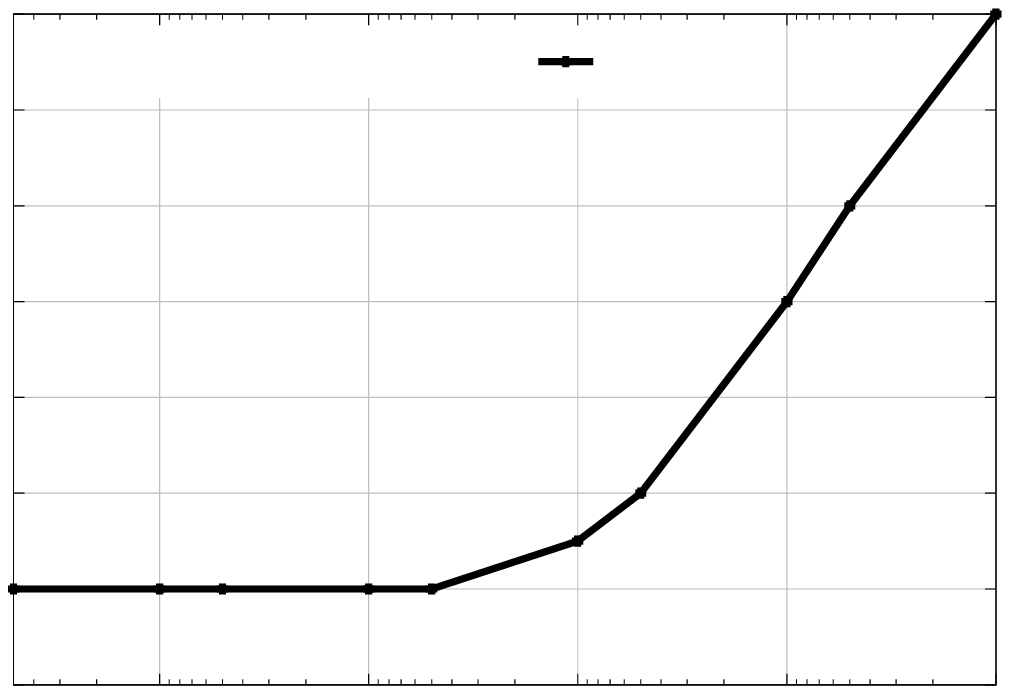}}}}
\end{minipage}
\caption{({Left}) The measured $\chi^2$ divergences, $\hat{\chi}^2_{\mathrm{meas},\fid}$, between the high-fidelity posterior $p$ and the Laplace approximation $q_{\fid}$ to each surrogate posterior $p_{\fid}$.  ({Right}) The selected fidelity for the number of elements $(\fid^*)^{-1}$ from the optimization \eqref{eq:optimization} as the tolerance $\epsilon$ on the MSE changes.}
\label{fig:laplace_chi2_optfid}
\end{figure}

\begin{figure}
\begin{minipage}{0.49\textwidth}
{{\Large\resizebox{\columnwidth}{!}{\input{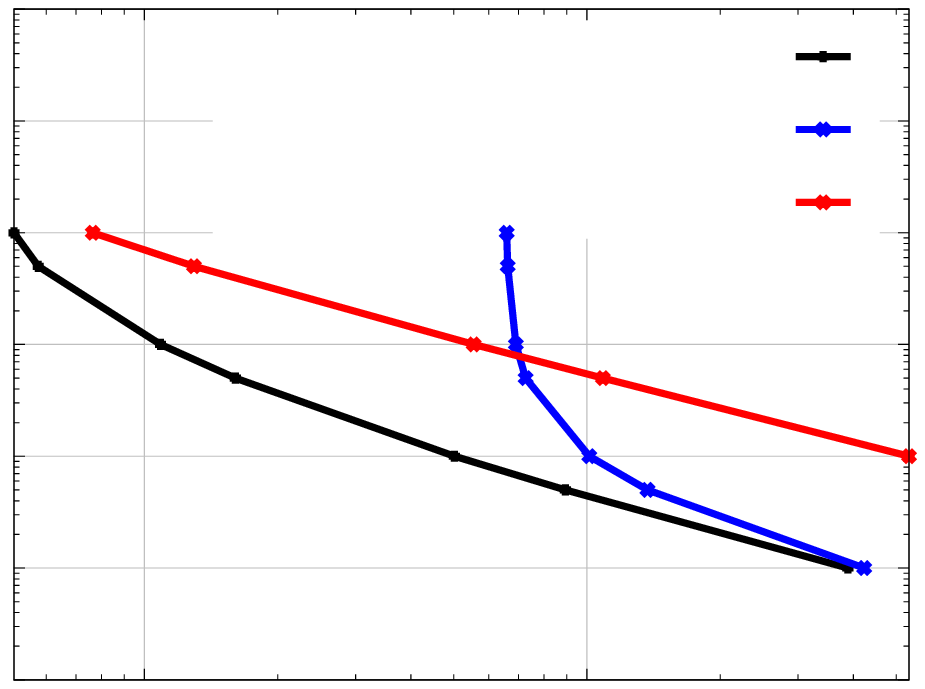}}}}
\end{minipage}
\hfill
\begin{minipage}{0.49\textwidth}
{{\Large\resizebox{\columnwidth}{!}{\input{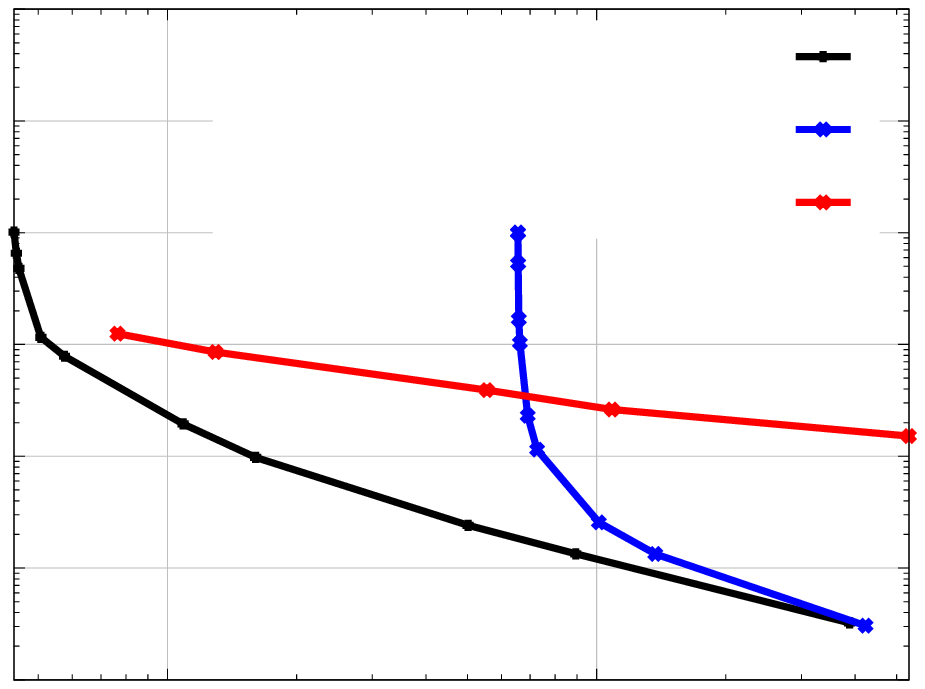}}}}
\end{minipage}
\caption{({Left}) The theoretical error tolerance $\epsilon$ against the total cost (seconds of CPU time) to fit the Laplace approximation $q_{\fid^*}$ of $p_{\fid^*}$ and draw $\samples^*$ samples.  ({Right}) The actual measured $\widehat{\mathrm{MSE}}_{\epsilon}$ against the total cost.  Note that the results shown in the left plot is an upper bound for the results shown in right plot by the bound \eqref{eq:mse_bound}.}
\label{fig:laplace_tradeoff}
\end{figure}

\subsection{Euler Bernoulli Problem}

In the second example we infer the effective stiffness of an Euler Bernoulli beam. The forward-model code is available on GitHub\footnote{\texttt{https://github.cim/g2s3-2018/labs}} and was developed by Matthew Parno as a part of the 2018 Gene Golub SIAM Summer School on ``Inverse Problems: Systematic Integration of Data with Models under Uncertainty”.  The rest of the setup of this problem is taken from Section 4.2 of \cite{PM}.

\subsubsection{Problem Setup}

Let $\spatialdomain = (0,1) \subset \R$ and $\paramdomain = \R^6$ and consider the PDE
\begin{equation}
        \frac{\partial^2}{\partial x^2} \left( E(x; \params) \frac{\partial^2}{\partial x^2} u(x; \params) \right) = g(x), \quad x \in \spatialdomain
\label{eq:eb_pde}
\end{equation}
with boundary conditions
 \[
     u(0;\params) = 0,\quad \frac{\partial u}{\partial x}(0;\params) = 0, \quad \frac{\partial^2 u}{\partial x^2}(1;\params) = 0, \quad \frac{\partial^3 u}{\partial x^3}(1;\params) = 0
\]
where $u:\spatialdomain \times \paramdomain \to \R$ is the displacement and $E:\spatialdomain \times \paramdomain \to \R$ is the effective stiffness of the beam.  The applied force $g(x)$ is taken to be $g(x) = 1$.  The effective stiffness $E(x;\params)$ is a smooth piecewise constant defined in the same way as the log-diffusivity \eqref{eq:log_diffusivity} but with $\theta_i$ replaced by $|\theta_i|$ for $i=1,\ldots,6$.  We discretize \eqref{eq:eb_pde} in the spatial domain $\spatialdomain$ with a mesh width $\fid > 0$  (i.e. $\fid^{-1} + 1$ grid points) using a second-order finite difference scheme and solve the resulting linear system of equations for the discretized solution $u_{\fid}$ at the grid points.

The parameter-to-observable map $\map_{\fid} : \paramdomain \to \R^{40}$ is the linear interpolant of the $\fid^{-1} + 1$ grid points evaluated at 40 equally spaced points in the spatial domain $(0,1)$
\[
\left( \map_{\fid}(\params) \right)_i = u_{\fid}\left( \frac{i-1}{39} \right),\quad i = 1,\ldots,40
\]
Note that we exclude the left end-point at $x=0$ since it is fixed by the boundary conditions.  We set the high-fidelity map to be $\map = \map_{H}$ with $H^{-1} + 1 = 256$ grid points and for the surrogate maps we again consider $\fid^{-1} + 1 = 8,12,16,\ldots,64$.

\subsubsection{Setup of the inverse problem}

A single observation $\obs = \map(\params_{\mathrm{truth}}) + \noisevar \in \R^{40}$ is generated where $\params_{\mathrm{truth}} = (1,\ldots,1)^T \in \R^6$ and $\noisevar \sim N(\zeros,\ \noise)$ with noise covariance $\noise = 5.623 \times 10^{-4} \eye_{40\times 40}$.  The added noise now corresponds to approximately 5\% of the true solution $u$ at the right endpoint $x = 1$.  The prior is again a Gaussian with mean $\prmean = (1,\ldots,1)^T \in \R^6$ and covariance $\prcov = 1.778 \times 10^{-2}  \eye_{6 \times 6} \in \R^{6\times 6}$.  For the test function we use the same test function \eqref{eq:testfun} from the steady-state heat problem.

\subsubsection{Results}

We again fit a Laplace approximation to each surrogate posterior $p_{\fid}$ using Newton-CG with the gradient and Hessian computed by second-order finite difference approximations.  The total number of model evaluations is $M = 1800$ at each fidelity.  The cost function has the form $\cost(\fid) = c_0 + c_1 /\fid$ (linear in $\fid^{-1}$ because the system of linear equations from the discretization is sparse) and the surrogate accuracy has the form $\acc(\fid) = a_1 h^2$ from the second-order finite difference spatial discretization.

We use the $\chi^2$ divergence estimator $\hat{\chi}^2_{\fid,10^5}$ from \eqref{eq:chi2estimate} and average the results over $N_1 = 100$ independent trials to obtain the measured value $\hat{\chi}^2_{\mathrm{meas},\fid}$ as in \eqref{eq:chi2estimate} for each surrogate map $h^{-1} + 1= 8,12,16\ldots,64$.  We then use these measured values to fit the curve $\tilde{K}_0 e^{K_1 h^2}$.  Figure~\ref{fig:eb_chi2_optfid} shows the results.  Observe that the $\chi^2$ divergence quickly levels off again.

The fidelity and sample size $(\fid^*, \samples^*)$ are found using a brute-force search and Figure~\ref{fig:eb_chi2_optfid} shows the selected number of grid points $(\fid^*)^{-1} + 1$ as a function of the MSE tolerance $\epsilon$.  When the tolerance is small the selected fidelity is the highest fidelity since we do not consider any surrogate models with $h^{-1} + 1$ between 64 and 256.  Figure~\ref{fig:eb_tradeoff} shows the theoretical optimal cost and error trade-off for $\hat{f}_{\fid^*,\samples^*}$.  We estimated the true value $\E{f}{p}$ using $\hat{f}_{H,10^5}$ with $N_2 = 100$ independent trials using equation~\eqref{eq:truth_estimate} and the MSE was estimated with $N_3 = 2500$ independent trials using equation~\eqref{eq:mse_estimate}.  Here the lowest-fidelity surrogate model corresponds to $\fid_0 = 16$.  From the plot we can observe an order of magnitude speedup for moderate tolerances where we do not need to use a high-fidelity model to fit the Laplace approximation.  Also note that the theoretical trade-off is an upper bound but the shape matches closely with the measured results.

\begin{figure}
\begin{minipage}{0.49\textwidth}
{{\Large\resizebox{\columnwidth}{!}{\input{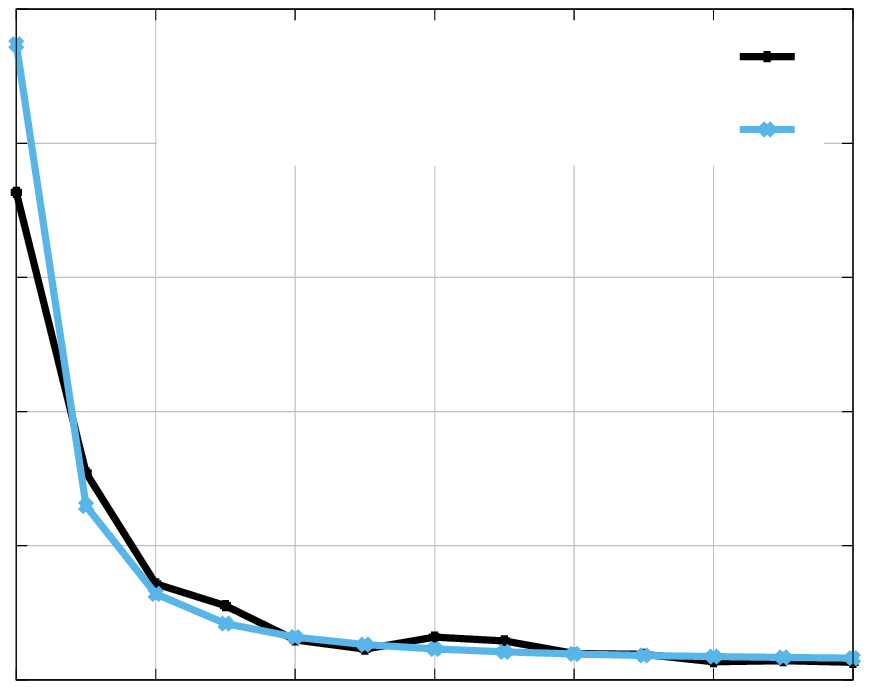}}}}
\end{minipage}
\hfill
\begin{minipage}{0.49\textwidth}
{{\Large\resizebox{\columnwidth}{!}{\input{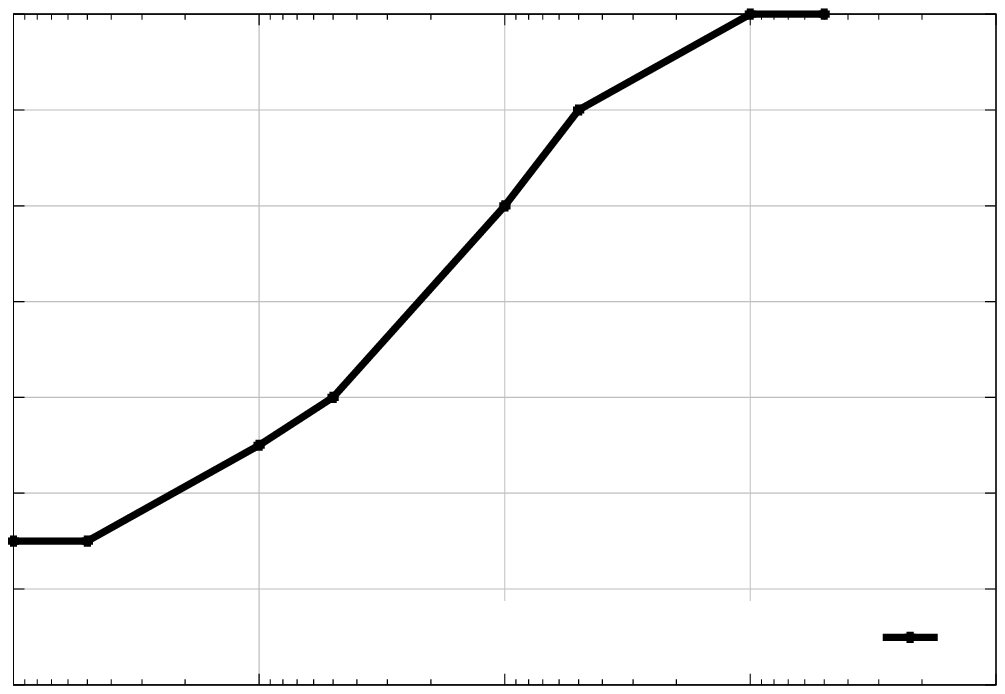}}}}
\end{minipage}
\caption{({Left}) The measured $\chi^2$ divergences, $\hat{\chi}^2_{\mathrm{meas},\fid}$, between the high-fidelity posterior $p$ and the Laplace approximation $q_{\fid}$ to each surrogate posterior $p_{\fid}$.  ({Right}) The selected fidelity for the number of grid points $(\fid^*)^{-1} + 1$ from the optimization \eqref{eq:optimization} as the tolerance $\epsilon$ on the MSE changes.}
\label{fig:eb_chi2_optfid}
\end{figure}

\begin{figure}
\begin{minipage}{0.49\textwidth}
{{\Large\resizebox{\columnwidth}{!}{\input{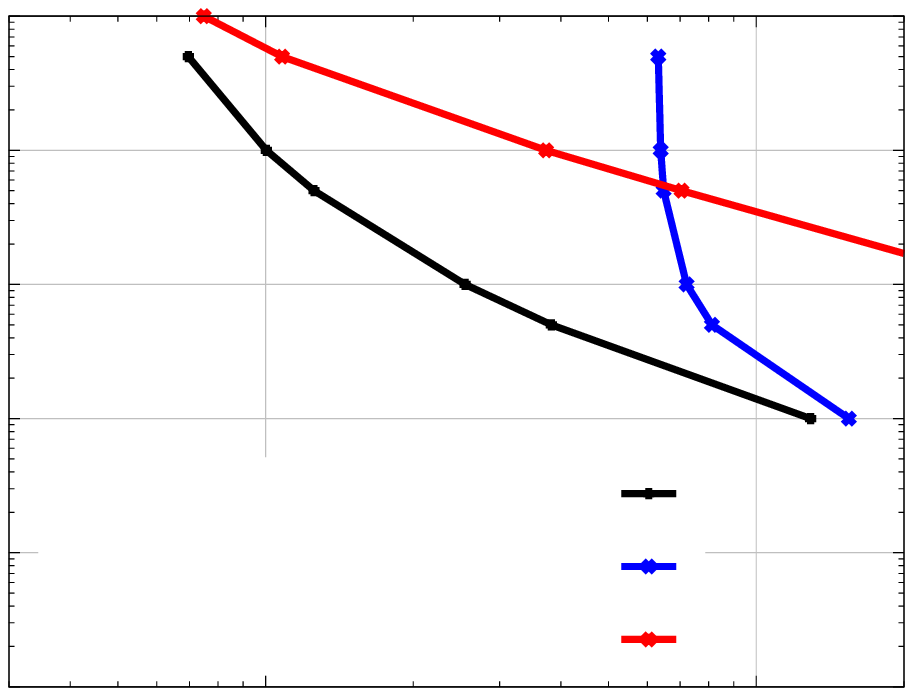}}}}
\end{minipage}
\hfill
\begin{minipage}{0.49\textwidth}
{{\Large\resizebox{\columnwidth}{!}{\input{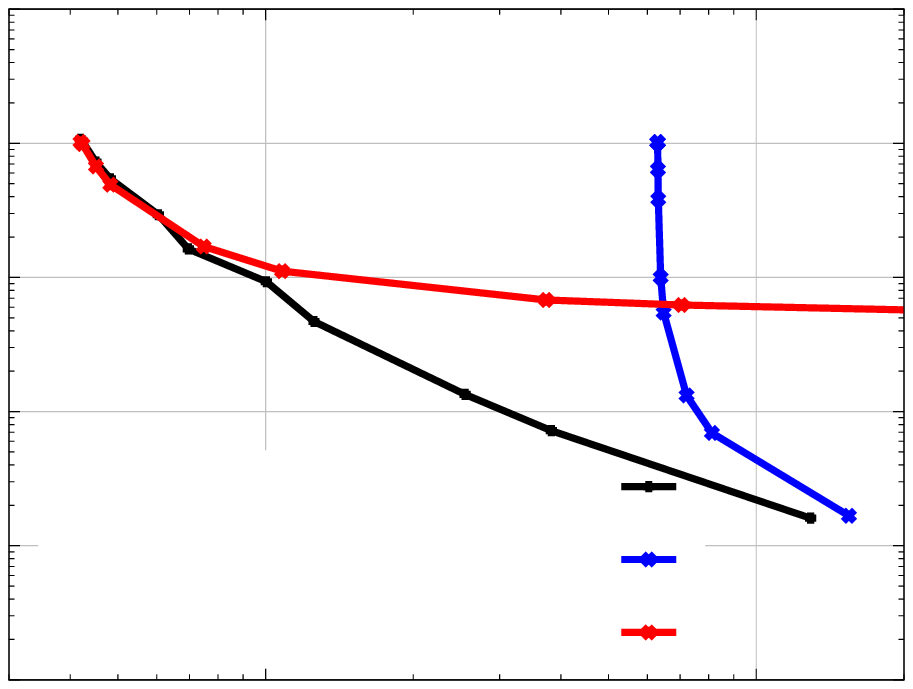}}}}
\end{minipage}
\caption{({Left}) The theoretical error tolerance $\epsilon$ against the total cost (seconds of CPU time) to fit the Laplace approximation $q_{\fid^*}$ of $p_{\fid^*}$ and draw $\samples^*$ samples.  ({Right}) The actual measured $\widehat{\mathrm{MSE}}_{\epsilon}$ against the total cost.  Note that the results in the left plot are upper bounding the results in the right plot by the bound \eqref{eq:mse_bound}.}
\label{fig:eb_tradeoff}
\end{figure}

\section*{Acknowledgements}
The authors acknowledge support of the National Science Foundation under Grant No.~1761068 and Grant No.~1901091. The first author was supported in part by the Research Training Group in Modeling and Simulation funded by the National Science Foundation via grant RTG/DMS – 1646339. The second author also acknowledges support from the AFOSR under Award Number FA9550-21-1-0222 (Dr.~Fariba Fahroo).

\bibliography{references}
\bibliographystyle{abbrv}

\appendix

\section{Proof of Lemma~\ref{lemma:subgaussian}}
\label{appdx:proof}

  \begin{proof}
    Suppose that $\bfx$ is a sub-Gaussian random vector and consider the matrix to be a multiple of the identity, $\bA = \alpha \eye$ with $\alpha > 0$.  We now only need to show that there exists an $\alpha > 0$ such that for all $\bfm \in \mathbb{R}^{\dimx}$
    \[
        \E{ \exp\left( \alpha \|\bfx - \bfm\|^2 \right) }{\pi} = \E{ \exp\left( (\bfx - \bfm)^T\bA (\bfx - \bfm) \right)}{\pi} < \infty\,.
    \]
    Since $\|\bfv + \bfw\|^2 \le 2\|\bfv\|^2 + 2\|\bfw\|^2$ by the triangle inequality and the fact that $(a+b)^2 \le 2a^2 + 2b^2$, we get the upper bound
    \[
        \E{\exp\left( \alpha \|\bfx - \bfm\|^2 \right) }{\pi} \le \E{ \exp \left( 2\alpha \|\bfm\|^2 + 2\alpha \|\bfx\|^2 \right) }{\pi} =  \exp\left(2\alpha \|\bfm\|^2 \right) \E { \exp\left( 2\alpha \|\bfx\|^2 \right)}{\pi}\,.
    \]
    Therefore, we only need to find $\alpha > 0$ such that
    \[
        \E{ \exp\left( 2\alpha \|\bfx\|^2 \right) }{\pi}< \infty\,.
    \]
    We now use the assumption that $\bfx$ is sub-Gaussian by taking the marginals
    \begin{align*}
        \E{ \exp\left( 2\alpha \|\bfx\|^2 \right) }{\pi}
        &= \E{ \exp\left( 2\alpha \sum_{i=1}^{\dimx} x_i^2 \right) }{\pi}\\
        &= \E{ \exp\left( 2\alpha \sum_{i=1}^{\dimx} |\cbv_i^T \bfx|^2 \right) }{\pi}\\
        &= \E{ \prod_{i=1}^{\dimx} \exp\left( 2\alpha |\cbv_i^T \bfx|^2 \right) }{\pi}\,,\\
    \end{align*}
    where $\cbv_i$ is the $i$-th canonical unit vector.
    We proceed by induction on the dimension $\dimx$ and repeatedly use the Cauchy-Schwarz inequality to show that this expectation is finite.  When $\dimx=1$, take $\alpha_1$ such that $\frac{1}{\sqrt{2\alpha_1}} > \|\bfx\|_{\psi_2}$ so that
    \[
        \E{ \exp\left( 2\alpha_1 |\cbv_1^T \bfx|^2 \right) }{\pi}
        = \E{  \exp\left( \frac{|\cbv_1^T \bfx|^2}{(1/\sqrt{2\alpha_1})^2} \right) }{\pi}
        \le 2\,.
    \]
    Note that since $\bfx$ is sub-Gaussian $\|\bfx\|_{\psi_2} < \infty$ we can indeed find an $\alpha_1 > 0$ to satisfy the inequality.  Now suppose that for dimension $\dimx-1$ there exists an $\alpha_{\dimx-1}$ such that
    \[
        \E{ \prod_{i=1}^{\dimx-1} \exp\left( 2\alpha_{\dimx-1} |\cbv_i^T \bfx|^2 \right)}{\pi}
        = C_{\dimx-1} < \infty\,.
    \]
    By using the Cauchy-Schwarz inequality, we get that
        \begin{align*}
        \E{ \prod_{i=1}^{\dimx} \exp\left( 2\alpha_{\dimx} |\cbv_i^T \bfx|^2 \right) }{\pi}
        & \le  \E{ \prod_{i=1}^{\dimx-1} \exp\left( 4\alpha_{\dimx} |\cbv_i^T \bfx|^2 \right) }{\pi}^{1/2}
          \E{  \exp\left( 4\alpha_{\dimx} |\cbv_{\dimx}^T \bfx|^2 \right) }{\pi}^{1/2}\,.
    \end{align*}
    Taking $\alpha_{\dimx} \le \alpha_{\dimx-1}/2$ gives
    \[
        \E{ \prod_{i=1}^{\dimx-1} \exp\left( 4\alpha_{\dimx} |\cbv_i^T \bfx|^2 \right) }{\pi}^{1/2}
        \le \E{ \prod_{i=1}^{\dimx-1} \exp\left( 2\alpha_{\dimx-1} |\cbv_i^T \bfx|^2 \right) }{\pi}^{1/2}
        = C_{\dimx-1}^{1/2}\,.
    \]
    Taking $\alpha_{\dimx}$ such that $\frac{1}{\sqrt{4\alpha_{\dimx}}} > \|\bfx\|_{\psi_2}$ gives
    \[
        \E{  \exp\left( 4\alpha_{\dimx} |\cbv_{\dimx}^T \bfx|^2 \right) }{\pi}^{1/2}
        \le \E{  \exp\left( \frac{ |\cbv_{\dimx}^T \bfx|^2}{(1/\sqrt{4\alpha_{\dimx}})^2} \right) }{\pi}^{1/2}
        \le \sqrt{2}\,.
    \]
    Thus, take $\alpha_{\dimx} < \frac{1}{4}\min \{2\alpha_{\dimx-1},\ \|\bfx\|_{\psi_2}^{-2} \}$, so that
    \[
        \E{ \prod_{i=1}^{\dimx} \exp\left( 2\alpha_{\dimx} |\cbv_i^T \bfx|^2 \right) }{\pi}
        \le \sqrt{2C_{\dimx-1}} < \infty\,.
    \]
    Since the dimension is finite, we know that we will always be able to take $\alpha_{\dimx} > 0$.  Setting $\alpha = \alpha_{\dimx}$, shows the first direction of the lemma.\\

    For the converse suppose that there exists a symmetric positive-definite matrix $\bA \succ 0$ so that for all vectors $\bfm$
    \[
        \E{ \exp \left( (\bfx - \bfm)^T \bA (\bfx - \bfm) \right) }{\pi} < \infty\,.
    \]
    In particular, for $\bfm = 0$
    \[
        \E{ \exp \left( \bfx^T \bA \bfx \right)}{\pi} = C < \infty\,.
    \]
    For any $\bfv \in S^{\dimx-1}$, we have that
    \begin{align*}
        \E{ \exp\left( \frac{|\bfv^T\bfx|^2}{t^2} \right) }{\pi}
        \le \E {\exp\left( \frac{\|\bfx \|^2}{t^2} \right) }{\pi}\,,
    \end{align*}
    since $|\bfv^T\bfx| \le \|\bfv \|  \|\bfx \|$.  Also, since the minimum eigenvalue satisfies $\lambda_{\min}^{\bA} \le \frac{\bfx^T \bA \bfx}{\|\bfx\|^2}$ for all $\bfx \neq 0$, we get
    \[
        \E{ \exp\left( \frac{\|\bfx \|^2}{t^2} \right) }{\pi}
        \le \E{ \exp\left( \frac{\bfx^T \bA \bfx}{\lambda_{\min}^{\bA} t^2} \right)}{\pi}
        = \E{ \left\{ \exp\left( \bfx^T \bA\bfx \right) \right\}^{1/\lambda_{\min}^{\bA} t^2} }{ \pi }\,.
    \]
    If $\lambda_{\min}^{\bA} t^2 > 1$, then the function
    \[
        g(x) = x^{1/(\lambda_{\min}^{\bA} t^2)}
    \]
    is concave and increasing in $x$.  By Jensen's inequality, we obtain
    \[
        \E{ \left\{ \exp\left( \bfx^T \bA \bfx \right) \right\}^{1/\lambda_{\min}^{\bA} t^2} }{\pi}
        \le \E{ \exp\left( \bfx^T \bA \bfx \right)}{\pi}^{1/\lambda_{\min}^{\bA } t^2}
        = C^{1/\lambda_{\min}^{\bA} t^2}\,.
    \]
    Setting $C^{1/\lambda_{\min}^{\bA} t^2} \le 2$ and solving for $t$ gives
    \[
        t \ge \sqrt{\frac{\log C}{\lambda_{\min}^{\bA} \log 2}}\,.
    \]
    Since this inequality holds for every $\bfv \in S^{\dimx-1}$ we know that $\|\bfx\|_{\psi_2} < \infty$ and hence $\bfx$ is sub-Gaussian.
\end{proof}

\end{document}